\documentclass[final,leqno,onefignum,onetabnum]{siamltex1213}
\usepackage{bbm}
\usepackage{amsmath,amsfonts,amssymb,enumerate}
\usepackage{color}
\usepackage{lscape}
\usepackage{latexsym}
\usepackage{multirow}
\usepackage{rotating}
\usepackage{threeparttable}
\usepackage{graphicx}
\usepackage{algpseudocode,algorithm}
\usepackage[margin=1in]{geometry}
\newtheorem{example}[theorem]{Example}
\newtheorem{remark}[theorem]{Remark}
\newtheorem{hypothesis}{H}
\graphicspath{{figs/}}
\definecolor{doushalv}{rgb}{0.78,0.93,0.80}

\title{On optimal solutions of the constrained $\ell_0$ regularization and its penalty problem\thanks{This research is supported in part by Guangdong Provincial Government of China through the ``Computational Science Innovative Research Team'' program, by the Natural Science Foundation of China under grants 11501584, and by the
Natural Science Foundation of Guangdong Province under grants 2014A030310332 and 2014A030310414.}}

\author {Na Zhang \footnotemark[3]
\and
Qia Li
\footnotemark[2]}

\begin{document}

\maketitle
\renewcommand{\thefootnote}{\fnsymbol{footnote}}
\footnotetext[3]{Department of Applied Mathematics, College of Mathematics and Informatics, South China Agricultural University, Guangzhou 510642, P. R. China. }
\footnotetext[2]{Guangdong Province Key Laboratory of Computational Science, School of Data and  Computer Sciences, Sun Yat-sen University, Guangzhou 510275, P. R. China.
(\email{liqia@mail.sysu.edu.cn}).
 Questions, comments, or corrections to this document may be directed to that email address.}

\begin{abstract}
\normalsize
The constrained $\ell_0$ regularization  plays an important role in sparse reconstruction. A widely used approach for solving this problem is the penalty method, of which the least square penalty problem is a special case. However, the connections between global minimizers of the constrained $\ell_0$ problem and its penalty problem have never been studied in a systematic way. This work provides a comprehensive investigation on optimal solutions of these two problems and their connections. We give detailed descriptions of optimal solutions of the two problems, including existence, stability with respect to the parameter, cardinality and strictness. In particular, we find that the optimal solution set of the penalty problem is piecewise constant with respect to the penalty parameter. Then we analyze in-depth the relationship between optimal solutions of the two problems. It is shown that, in the noisy case the least square penalty problem probably has no common optimal solutions with the constrained $\ell_0$ problem for any penalty parameter. Under a mild condition on the penalty function, we establish that the penalty problem has the same optimal solution set as the constrained $\ell_0$ problem when the penalty parameter is sufficiently large. Based on the conditions, we further propose exact penalty problems for the constrained $\ell_0$ problem. Finally, we present a numerical example to illustrate our main theoretical results.
\end{abstract}



\pagestyle{myheadings}
\thispagestyle{plain}
{
\section{Introduction}
Signal  reconstruction is usually formulated as a linear inverse problem
\begin{equation}\label{inverP}
b=Ax+\eta,
 \end{equation}
where $x\in\mathbb{R}^n$ is the true signal, $A$ is an $m\times n$ real matrix, $b\in\mathbb{R}^m$ is the observed data and $\eta\in\mathbb{R}^m$ represents the random noise. Over the last  decade,
sparse reconstruction has received great attentions \cite{Bruckstein-Donoho-Elad:2009, Cand2006Sparsity, Donoho:2006,Nikolova-Ng-Zhang-Ching:2008}. In general, the task of sparse reconstruction is to find a sparse $x$ from \eqref{inverP}.
Obviously, the most natural measure of sparsity is the $\ell_0$-norm, defined at $x\in\mathbb{R}^n$ as
$$
\|x\|_0:=\sharp\supp(x),$$
where $\sharp S$ stands for the cardinality of the set $S$ and $\supp(x)$ denotes the support of $x$, that is $\supp(x):=\{i\in\{1,2,\dots, n\}: x_i\neq 0\}$. When the noise level of $\eta$ is known, we consider the following constrained $\ell_0$ regularization problem
\begin{equation}\label{prob:l2-constrain}
\min\{\|x\|_0: \|Ax-b\|_2\le \sigma, x\in\mathbb{R}^n\},
\end{equation}
where $\|\cdot\|_2$ denotes the $\ell_2$-norm, and $\sigma\ge0$ is the noise level. Specially, in the noiseless case, i.e., when $\sigma=0$, problem \eqref{prob:l2-constrain} reduces to seeking for the sparsest solutions from a linear system as follows
\begin{equation}\label{prob:lineq}
\min\{\|x\|_0: Ax=b, x\in\mathbb{R}^n\}.
\end{equation}

Finding a global minimizer of problem \eqref{prob:l2-constrain} is known to be combinational and NP-hard in general \cite{Davis1997Adaptive,Natarajan1995Sparse,Tropp:2006} as it involves the $\ell_0$-norm. The number of papers dealing with the $\ell_0$-norm is large and several types of numerical algorithms have been adopted to approximately solve the problem. In this paper, we mainly focus on the penalty methods. Before making a further discussion on the penalty methods, we briefly review other two types of approaches, greedy pursuit methods and relaxation of the $\ell_0$-norm. Greedy pursuit methods, such as Orthogonal Matching Pursuit \cite{Tropp-Gilbert:2007} and CoSaMP \cite{Needell-Tropp:2009}, adaptively update the support of $x$ by exploiting the residual and then solve a least-square problem on the support in each iteration. This type of algorithms are widely used in practice due to their relatively fast computational speed. However, only under very strict conditions, can they be shown theoretically to produce desired numerical solutions to problem \eqref{prob:l2-constrain}. Relaxation methods  first replace the $\ell_0$-norm by other continuous cost functions, including convex functions \cite{BergFriedlander:SIAMO:2011} and nonconvex functions \cite{Chartrand:IEEE-Letter:07}, such as the bridge functions $\|\cdot\|^q_q$ with $0<q<1$ \cite{Lu:2014}, capped-$\ell_1$ functions \cite{Fan2002Variable},  minmax concave functions \cite{Zhang:ANNALS:2010} and so on. Then the majorization-minimization strategy is applied to the relaxation problem, which leads to solving a constrained reweighted $\ell_1$ or $\ell_2$ subproblems \cite{Candes-Wakin-Boyd:JFAA:08, Lu:2014} in each step. The cost functions of the iterates are shown to be descent \cite{Bissantz2009Convergence}, while the convergence of the iterates themselves are generally unknown.

The penalty method is a classical and important approach for constrained optimization. For an overview on the penalty method, we refer the readers to the book \cite{Nocedal-Wright:book}. An extensively used penalty method for problem \eqref{prob:l2-constrain} is to solve the following least square  penalty problem
\begin{equation}\label{prob:l2-unconstrain}
\min\{\|x\|_0+\frac{\lambda}{2} \|Ax-b\|^2_2: x\in\mathbb{R}^n\},
\end{equation}
where $\lambda>0$ is a penalty parameter.  As problem \eqref{prob:l2-unconstrain} still involve the $\ell_0$-norm, the greedy pursuit and relaxation methods mentioned before are applicable for approximately solving the problem. However, it has relatively better structure than problem \eqref{prob:l2-constrain} from the standpoint of algorithmic design. More precisely, the proximal operator of $\ell_0$-norm has a closed form, while the least square penalty term in problem \eqref{prob:l2-unconstrain} is differentiable with a Lipschitz continuous gradient. Therefore, it can be directly and efficiently solved by proximal-gradient type algorithms  \cite{Bonettini2015New,Figueiredo2008Gradient, Wright-Nowak-Figueiredo:IEEESP-09} or primal-dual active set methods \cite{Ito-Kunisch:2014,jiao-jin-lvv:acha2016}. One specific proximal-gradient type algorithm for problem \eqref{prob:l2-unconstrain} is the iterative hard thresholding (IHT) \cite{Blumensath-Davies:2008}. In general, it is shown in \cite{attouch-Convergence, Blumensath-Davies:2008} that for arbitrary initial point, the iterates generated by the algorithm converge to a local minimizer of problem \eqref{prob:l2-unconstrain}. Under certain conditions, the IHT converges with an asymptotic linear rate. The structure of the least square penalty problem makes it easier to  develop convergent numerical algorithms.

The penalty parameter $\lambda$  balances  the magnitude of the residual and the sparsity of solutions to  problem \eqref{prob:l2-unconstrain}. An important issue is that wether problem \eqref{prob:l2-constrain} and  \eqref{prob:l2-unconstrain} share a common global minimizer for some penalty parameter $\lambda$. It seems intuitive that there exists certain connections between these two problems. However, to the best of our knowledge, there is very little theoretical results on penalty methods for problem \eqref{prob:l2-constrain} in general. Recently, Chen, Lu and Pong study the penalty methods for a class of constrained non-Lipschitz optimization \cite{Chen-Lu:penalty2014}. The non-Lipschitz cost functions considered in \cite{Chen-Lu:penalty2014} does not contain the $\ell_0$-norm and the corresponding results can not be generalized to the constrained $\ell_0$ regularization problems. Nikolova in \cite{Nikolova:2013} gives a detailed mathematical description of optimal solutions to problem \eqref{prob:l2-unconstrain} and further in \cite{Nikolova:ACHA2016} establishes their connections with optimal solutions to the following optimization problem
\begin{equation}\label{prob:l2-K}
\min\{ \|Ax-b\|^2_2:\|x\|_0\leq D, x\in\mathbb{R}^n\},
\end{equation}
where $D>0$ is a given positive integer. Problem \eqref{prob:l2-K} also involves the $\ell_0$-norm and is usually called the sparsity constrained minimization problem. However, optimal solutions of problem \eqref{prob:l2-constrain} and their connections with problem \eqref{prob:l2-unconstrain} have not been studied in these works.

This paper is devoted to describing global minimizers of a class of constrained $\ell_0$ regularization problems and the corresponding penalty methods. As in \cite{Teboulle:asympt2003,Rockafellar2004Variational}, we use ``optimal solutions'' for global minimizers and ``optimal solution set'' for the set of all global minimizers in the remaining part of this paper. The class of constrained $\ell_0$ regularization problems we considered has the following form
\begin{equation}\label{prob:constrain}
\min\{\|x\|_0: \|Ax-b\|_p\le \sigma, x\in\mathbb{R}^n\},
\end{equation}
where $p=1$ or $2$. Obviously, problem \eqref{prob:constrain} has two cases, the $\ell_2$ constrained problem \eqref{prob:l2-constrain} and the $\ell_1$ constrained problem as follows
\begin{equation}\label{prob:l1-constrain}
\min\{\|x\|_0: \|Ax-b\|_1\le \sigma, x\in\mathbb{R}^n\}.
\end{equation}
 The analysis of these two constrained problems and their penalty methods share many common points and thus we shall discuss them in a unified way. Moreover, we introduce a class of scalar penalty functions $\phi:\mathbb{R}_+\rightarrow \mathbb{R}_+$  satisfying the following blanket assumption, where $\mathbb{R}_+:=[0, +\infty)$.
\begin{hypothesis}\label{hypo:1}
The function $\phi:\mathbb{R}_+\rightarrow\mathbb{R}_+$ is continuous, nondecreasing with $\phi(0)=0$.
\end{hypothesis}

Then the penalty problems for \eqref{prob:constrain} can be in general formulated as
\begin{equation}\label{prob:penalty}
\min\{\|x\|_0+\lambda\phi(\|Ax-b\|_p): x\in\mathbb{R}^n\},
\end{equation}
where $\lambda>0$ is a penalty parameter.
In the case of $p=2$, when we set the penalty function $\phi(z):= \frac{z^2}{2}$ for $z\ge 0$, the general penalty problem \eqref{prob:penalty} reduces to the specific problem \eqref{prob:l2-unconstrain}. We shall give a comprehensive study on connections between optimal solutions of problem \eqref{prob:constrain} and problem \eqref{prob:penalty}.

Our main contributions are summarized below.
\begin{itemize}
\item We study in depth the existence, stability and cardinality of solutions to problems \eqref{prob:constrain} and \eqref{prob:penalty}. In particular, we prove that the optimal solution set of problem \eqref{prob:penalty} is piecewise constant with respect to $\lambda$. More precisely, we present a sequence of real number $\lambda_{-1}=+\infty>\lambda_0>\lambda_1>\cdots>\lambda_{K}=0$ where $K\leq \mathrm{rank}(A)$. We claim that for all $j=0,1,\dots,K$, problems \eqref{prob:penalty} with different $\lambda\in (\lambda_{j},\lambda_{j-1})$ share completely the same  optimal solution set.
\item We clarify the relationship between optimal solutions of problem \eqref{prob:l2-constrain} and those  of problem \eqref{prob:l2-unconstrain}. We find that, when $\sigma>0$ problem \eqref{prob:l2-unconstrain} probably {\bf {never}} has a common optimal solution with problem \eqref{prob:l2-constrain} for any $\lambda>0$. This means that the penalty problem \eqref{prob:l2-unconstrain} may fail to be an exact penalty formulation of problem \eqref{prob:l2-constrain}. A numerical example is presented to illustrate this fact in Section \ref{sec:exp}.
\item We show that under a mild condition on the penalty function $\phi$,  problem \eqref{prob:penalty} has exactly the same optimal solution set as that of problem \eqref{prob:constrain} once $\lambda>\lambda^*$ for some $\lambda^*>0$. In particular, for $p=2$ we present a penalty function $\phi$ with which the penalty term in problem \eqref{prob:penalty} is differentiable with a Lipschitz continuous gradient. This penalty function enables us to design innovative numerical algorithms for problem \eqref{prob:penalty}.
\end{itemize}

The remaining part of this paper is organized as follows. In Section \ref{sec:exist} we show the existence of optimal solutions to problems \eqref{prob:constrain} and \eqref{prob:penalty}. The stability for problems \eqref{prob:constrain} and \eqref{prob:penalty}  are discussed in Section \ref{sec:sta}. We analyze in Section \ref{sec:relation} the relationship between optimal solutions of problems \eqref{prob:l2-constrain} and \eqref{prob:l2-unconstrain}. Section \ref{sec:equal} establishes a sufficient condition on $\phi$, which ensures problems \eqref{prob:penalty} and \eqref{prob:constrain} have completely the same optimal solution set when the parameter $\lambda$
is sufficiently large. We further investigate the cardinality and strictness of optimal solutions to problems \eqref{prob:constrain} and \eqref{prob:penalty} in Section \ref{sec:minimizer}. The theoretical results are illustrated by numerical experiments in Section \ref{sec:exp}. We conclude this paper in Section \ref{sec:con}.

\section{Existence of optimal solutions to problems \eqref{prob:constrain} and \eqref{prob:penalty}}\label{sec:exist}
 This section is devoted to the existence of optimal solutions to problems \eqref{prob:constrain} and \eqref{prob:penalty}. The former issue is trivial, while the latter issue requires careful analysis. After a brief discussion on the existence of optimal solutions to problem \eqref{prob:constrain}, we shall focus on a constructive proof of the existence of optimal solutions to problem \eqref{prob:penalty}.

Since the objective function of problem \eqref{prob:constrain} is proper, lower semi-continuous and has finite values, it is clear that problem \eqref{prob:constrain} has optimal solutions as long as the feasible region is nonempty. We present the result in the following theorem.  For simplicity, we denote by $G_\sigma$ the feasible region of problem \eqref{prob:constrain}, that is
\begin{equation}\label{def:G}
G_\sigma:=\{x:\|Ax-b\|_p\le \sigma, x\in\mathbb{R}^n\}.
\end{equation}

\begin{theorem}\label{thm:existConst}
The optimal solution set of problem \eqref{prob:constrain} is nonempty  if and only if $G_\sigma\neq \emptyset$, where $G_\sigma$ is defined by \eqref{def:G}.
\end{theorem}

In the rest of this section, we dedicate to discussing the existence of optimal solutions to problem \eqref{prob:penalty}.
Since the objective function of \eqref{prob:penalty} consists of two terms, we begin with defining several notations by alternating minimizing each term. To this end, for any positive integer $k$, we define
$$
\mathbb{N}_k:=\{1,2,\dots, k\} \mathrm{~~and~~} \mathbb{N}_{k}^0:=\{0,1,\dots, k\}.
$$
\begin{definition}\label{def:sRhoOmega}
Given $\phi$ satisfying H\ref{hypo:1} and $p\in\{1, 2\}$, the integer $L$, the sets $\{s_i\in \mathbb{N}: i \in\mathbb{N}_{L}^0\}$, $\{\rho_i\ge 0: i \in\mathbb{N}_{L}^0\}$, $\{\Omega_i\subseteq \mathbb{R}^n: i \in\mathbb{N}_{L}^0\}$ are defined by the following iteration
\begin{eqnarray*}
\mathrm{set~~}& i=0,\\ &\rho_0:=&\min\{\phi(\|Ax-b\|_p):x\in\mathbb{R}^n\},\\
&s_0:=&\min\{\|x\|_0: \phi(\|Ax-b\|_p)=\rho_0, x\in\mathbb{R}^n\},\\
&\Omega_0:=&\mathrm{arg}\min\{\|x\|_0: \phi(\|Ax-b\|_p)=\rho_0, x\in\mathbb{R}^n\},\\
\mathrm{while~~}&s_i>0 &\\
&\rho_{i+1}:=&\min\{\phi(\|Ax-b\|_p): \|x\|_0\le s_i-1, x\in\mathbb{R}^n\},\\
&s_{i+1}:=&\min\{\|x\|_0: \phi(\|Ax-b\|_p)=\rho_{i+1}, x\in\mathbb{R}^n\},\\
&\Omega_{i+1}:=&\arg\min\{\|x\|_0: \phi(\|Ax-b\|_p)=\rho_{i+1}, x\in\mathbb{R}^n\},\\
&i=&i+1,\\
\mathrm{end}&&\\
&L:=i.\\
\end{eqnarray*}
\end{definition}

We first show that Definition \ref{def:sRhoOmega} is well defined. It suffices to prove both the following two optimization problems
\begin{equation}\label{eq:subpConsL0}
\min\{\phi(\|Ax-b\|_p):\|x\|_0\le k, x\in\mathbb{R}^n\}
\end{equation}
and
\begin{equation}\label{eq:subpL0}
\min\{\|x\|_0: x\in S\}
\end{equation}
have optimal solutions for any $k\in \mathbb{N}_n^0$ and $\emptyset \neq S\subseteq\mathbb{R}^n$.
Clearly, problem \eqref{eq:subpL0} has an optimal solution since the objective function is piecewise constant and has finite values. The next lemma tells us that the optimal solution set of problem \eqref{eq:subpConsL0} is always nonempty for any $k\in\mathbb{N}_n^0$.  We denote by $\mathbf{0}_n$  the $n$-dimensional vector with $0$ as its components. For an $m\times n$ matrix $B$ and $\Lambda\subseteq \mathbb{N}_n$, let $B_\Lambda$ formed by the columns of $B$ with indexes in $\Lambda$. Similarly, for any $x\in\mathbb{R}^n$ and $\Lambda\subseteq \mathbb{N}_n$, we denote by $x_\Lambda$ the vector formed by the components of $x$ with indexes in $\Lambda$.

\begin{lemma}\label{lema:L0constSolu}
Given $\phi$ satisfying H\ref{hypo:1} and $p\in\{1, 2\}$,  for any $k\in\mathbb{N}_n^0$, the optimal solution set of problem \eqref{eq:subpConsL0} is nonempty.
\end{lemma}
\begin{proof}
As function $\phi$ is nondecreasing in $\mathbb{R}_+$, it suffices to show that the optimization problem
\begin{equation}\label{eq:subpConsL0-2}
\min\{\|Ax-b\|_p:\|x\|_0\le k, x\in\mathbb{R}^n\}
\end{equation}
has an optimal solution. We first prove that for any $\Lambda\subseteq \mathbb{N}_n$, the optimization problem
\begin{equation}\label{eq:subpConsL0-3}
\min\{\|Ax-b\|_p:\mathrm{supp}(x)\subseteq\Lambda, x\in\mathbb{R}^n\}
\end{equation}
has an optimal solution. Obviously, $x^*\in\mathbb{R}^n$ is an optimal solution to problem \eqref{eq:subpConsL0-3} if and only if $x^*_{\mathbb{N}_n\backslash\Lambda}=\mathbf{0}_{n-\sharp\Lambda}$ and
\begin{equation}\label{eq:subpConsL0-4}
x^*_{\Lambda}\in\arg\min\{\|A_{\Lambda}z - b\|_p: z\in\mathbb{R}^{\sharp\Lambda}\}.
\end{equation}
In the case of $p=2$, it is well-known that $x^*_{\Lambda}=A_{\Lambda}^{\dag}b$ is an optimal solution to problem \eqref{eq:subpConsL0-4}, where $A_\Lambda^\dag$ is the pseudo-inverse matrix of the matrix $A_\Lambda$.  When $p=1$, the optimal solution set of problem \eqref{eq:subpConsL0-4} is nonempty since $\|A_{\Lambda}\cdot-b\|_1$ is a piecewise linear function and bounded below, see Proposition 3.3.3 and 3.4.2 in \cite{Teboulle:asympt2003}. Therefore, problem \eqref{eq:subpConsL0-3} has an optimal solution for any $\Lambda\subseteq \mathbb{N}_n$. With the help of problem \eqref{eq:subpConsL0-3}, problem \eqref{eq:subpConsL0-2} can be equivalently written as
\begin{equation}\label{eq:subpConsL0-5}
\min\{\min\{\|Ax-b\|_p:\mathrm{supp}(x)\subseteq\Lambda, x\in\mathbb{R}^n\}:\Lambda\subseteq \mathbb{N}_n, \sharp\Lambda =k\}.
\end{equation}
Clearly, the optimal solution set of problem \eqref{eq:subpConsL0-5} is nonempty, therefore, problem \eqref{eq:subpConsL0-2} has an optimal solution. We then complete the proof.
\end{proof}

By Lemma \ref{lema:L0constSolu}, Definition \ref{def:sRhoOmega} is well defined. We present this result in the following proposition.

\begin{proposition}
 Given $\phi$ satisfying H\ref{hypo:1} and $p\in\{1, 2\}$, the integer $L$ and $s_i, \rho_i, \Omega_i$ for $i\in\mathbb{N}_{L}^0$ can be well defined by Definition \ref{def:sRhoOmega}.
\end{proposition}

We then provide some properties of $L$ and $s_i, \rho_i, \Omega_i$ for $ i \in\mathbb{N}_{L}^0$ defined by Definition \ref{def:sRhoOmega} in the following proposition.
\begin{proposition}\label{prop:rhoSOmega}
Given $\phi$ satisfying H\ref{hypo:1} and $p\in\{1, 2\}$, let $L$ and $s_i, \rho_i, \Omega_i$ for $i\in\mathbb{N}_{L}^0$ be defined by Definition \ref{def:sRhoOmega}. Then the following statements hold:
\begin{itemize}
\item [(i)] $0\le s_0\le \mathrm{rank}(A)$ and $0\le L\le \mathrm{rank}(A)$, where $\mathrm{rank}(A)$ denotes the rank of $A$.
\item [(ii)] $s_0>s_1>\dots>s_L=0$.
\item [(iii)] $0\le\rho_0<\rho_1<\dots<\rho_L=\phi(\|b\|_p)$.
\item [(iv)] $\Omega_L=\{\mathbf{0}_n\}$, $\Omega_i\neq \emptyset$ and $\Omega_i\cap \Omega_j=\emptyset$, for $i\neq j$, $i, j\in\mathbb{N}_L^0$.
\item [(v)] $\Omega_i=\{x: \|x\|_0=s_i, \phi(\|Ax-b\|_p)=\rho_i, x\in\mathbb{R}^n\}$ for $i \in\mathbb{N}_{L}^0$.
\end{itemize}
\end{proposition}
\begin{proof}
Items (ii), (iii), (iv) and (v) are clear by Definition \ref{def:sRhoOmega}. We dedicate to proving Item (i) following.

  Let $\Omega^*:=\arg\min\{\|Ax-b\|_p: x\in\mathbb{R}^n\}$. From the proof of Lemma \ref{lema:L0constSolu}, $\Omega^*\neq \emptyset$. Since $\phi$ is nondecreasing, $\phi(\|Ax-b\|_p)=\rho_0$ for any $x\in\Omega^*$. Let $x^*\in\Omega^*$, then $\phi(\|Ax^*-b\|_p)=\rho_0$. Let $\alpha_i\in\mathbb{R}^m$ denote the $i$th column of $A$, $i\in\mathbb{N}_n$. Since the rank of $A$ is $\mathrm{rank}(A)$, without loss of generality, we assume $\{\alpha_i: i\in\mathbb{N}_{\rank(A)}\}$ is the maximal linearly independent group of $\{\alpha_i: i\in\mathbb{N}_n\}$. Then there exists $\{a_i\in\mathbb{R}: i\in\mathbb{N}_{\rank(A)}\}$ such that $Ax^*=\sum_{i=1}^{\mathrm{rank}(A)}a_i\alpha_i$. We define $\tilde x\in\mathbb{R}^n$ as $\tilde x_i=a_i$ for $i\in\mathbb{N}_{\rank(A)}$ and $\tilde x_i=0$ for $i=\mathrm{rank}(A)+1,\dots, n$. It is obvious that $\|\tilde x\|_0\le \mathrm{rank}(A)$ and $Ax^*=A\tilde x$, implying $\tilde x\in\Omega^*$. Therefore $\phi(\|A\tilde x-b\|_p)=\rho_0$. Then by the definition of $s_0$, we have $s_0\le \|\tilde x\|_0\le \mathrm{rank}(A)$. From the procedure of the iteration in Definition \ref{def:sRhoOmega} and $s_L=0$, we obtain $L\le \mathrm{rank}(A)$.
  \end{proof}

With the help of Definition \ref{def:sRhoOmega}, the Euclid space $\mathbb{R}^n$ can be partitioned into $L+1$ sets: $\{x: \|x\|_0\ge s_0, x\in\mathbb{R}^n\}$, $\{x: s_{i+1}\le\|x\|_0\le s_i-1, x\in\mathbb{R}^n\}$, $i  \in\mathbb{N}_{L-1}^0$. Therefore, in order to establish the existence of optimal solutions to problem \eqref{prob:penalty}, it suffices to prove the existence of optimal solutions to problem
\begin{equation}\label{eq:subPartition1}
\min\{\|x\|_0+\lambda\phi(\|Ax-b\|_p): \|x\|_0\ge s_0, x\in\mathbb{R}^n\}
\end{equation}
and
\begin{equation}\label{eq:subPartition2}
\min\{\|x\|_0+\lambda\phi(\|Ax-b\|_p): s_{i+1}\le\|x\|_0\le s_i-1, x\in\mathbb{R}^n\}
\end{equation}
for all $i \in\mathbb{N}_{L-1}^0$. We present these desired results in the following lemma.
\begin{lemma}\label{lema:partitionSolution}
Given $\phi$ satisfying H\ref{hypo:1} and $p\in\{1, 2\}$, let $L$ and $s_i, \rho_i, \Omega_i$ for $i\in\mathbb{N}_{L}^0$ be defined by Definition \ref{def:sRhoOmega}. Then  for any $\lambda>0$, the following statements hold:
\begin{itemize}
\item [(i)] The optimal solution sets of problems \eqref{eq:subPartition1} and \eqref{eq:subPartition2} for $i  \in\mathbb{N}_{L-1}^0$ are not empty.
\item [(ii)] The optimal value of problem \eqref{eq:subPartition1} is $s_0+\lambda\rho_0$ and the optimal solution set to problem \eqref{eq:subPartition1} is $\Omega_0$.
\item [(iii)] The optimal value of problem \eqref{eq:subPartition2} is $s_{i+1}+\lambda\rho_{i+1}$ and the optimal solution set of problem \eqref{eq:subPartition2} is $\Omega_{i+1}$, for $i  \in\mathbb{N}_{L-1}^0$.
\end{itemize}
\end{lemma}
\begin{proof}
We only need to prove Items (ii) and (iii) since they imply Item (i).

We first prove Item (ii). It is obvious that restricted to the set $\{x: \|x\|_0\ge s_0, x\in\mathbb{R}^n\}$, the minimal value of the first term $\|x\|_0$ is $s_0$ and can be attained at any $x\in\Omega_0$. By Definition \ref{def:sRhoOmega}, the minimal value of $\phi(\|Ax-b\|_p)$ is $\rho_0$ and can be attained at any $x\in\Omega_0$. Therefore, the optimal value of problem \eqref{eq:subPartition1} is $s_0+\lambda\rho_0$. Let $\Omega^*$ be the optimal solution set of problem \eqref{eq:subPartition1}. Clearly, $\Omega_0\subseteq \Omega^*$. We then try to prove $\Omega^*\subseteq \Omega_0$. It suffices to prove $\|x\|_0=s_0$ for any $x\in\Omega^*$. If not, there exists  $x^*\in\Omega^*$ and $\|x^*\|_0\neq s_0$, then $\|x^*\|_0\ge s_0+1$. Then the objective function value at $x^*$ is no less than $s_0+1+\lambda\rho_0$ due to the definition of $\rho_0$, contradicting the fact that $x^*$ is an optimal solution of problem \eqref{eq:subPartition1}. Then we get Item (ii).
Item (iii) can be obtained similarly, we omit the details here.
\end{proof}

For convenient presentation, we define $f_i:\mathbb{R}_+\rightarrow \mathbb{R}$, for $i \in\mathbb{N}_{L}^0$, at $\lambda\ge 0$ as
\begin{equation}\label{def:fi}
f_i(\lambda):=s_i+\lambda\rho_i.
\end{equation}
Now, we are ready to show the existence of optimal solutions to problem \eqref{prob:penalty}.
\begin{theorem}\label{thm:exist}
 Given $\phi$ satisfying H\ref{hypo:1} and $p\in\{1, 2\}$, let $L$ and $s_i, \rho_i, \Omega_i$ for $i\in\mathbb{N}_{L}^0$ be defined by Definition \ref{def:sRhoOmega}. Let $f_i:\mathbb{R}_+\rightarrow\mathbb{R}$ for $i\in\mathbb{N}_L^0$ be defined by \eqref{def:fi}. Then, the optimal solution set of problem \eqref{prob:penalty} is nonempty for any $\lambda>0$. Furthermore, the following statements hold for any fixed $\lambda>0$:
 \begin{itemize}
 \item [(i)] The optimal value of problem \eqref{prob:penalty} is $\min\{f_i(\lambda): i \in\mathbb{N}_{L}^0\}$.
\item [(ii)] $\bigcup_{i\in\Lambda^*}\Omega_{i}$ is  the optimal solution set of problem \eqref{prob:penalty}, where $\Lambda^*:=\arg\min\{f_i(\lambda): i \in\mathbb{N}_{L}^0\}$.
    \end{itemize}
\end{theorem}
We omit the proof here since
it is a direct result of the fact that  $\mathbb{R}^n=\bigcup_{i=0}^{L-1}\{x:s_{i+1}\le\|x\|_0\le s_{i}-1, x\in\mathbb{R}^n\}\bigcup\{x:\|x\|_0\ge s_0, x\in\mathbb{R}^n\}$ and Lemma \ref{lema:partitionSolution}.

\begin{remark}\label{remark:L=0}
 According to Definition \ref{def:sRhoOmega}, if $L=0$, then $s_0=s_L=0$ and $\Omega_0=\Omega_L=\{\mathbf{0}_n\}$.  Then, by Theorem \ref{thm:exist}, if $L=0$, the optimal value and the optimal solution set of problem \eqref{prob:penalty} are $\rho_0\lambda$ and $\{\mathbf{0}_n\}$ respectively for any $\lambda>0$.
\end{remark}
\section{Stability for problems \eqref{prob:constrain} and \eqref{prob:penalty}}\label{sec:sta}
In this section, we study the stability for problems \eqref{prob:constrain} and \eqref{prob:penalty}, including behaviors of the optimal values and optimal solution sets with respect to changes in the corresponding parameters. We prove that, the optimal value of problem \eqref{prob:constrain} and that of problem \eqref{prob:penalty} change piecewise constantly and piecewise linearly respectively as the corresponding parameters vary. Moreover, we prove that the optimal solution set of problem   \eqref{prob:penalty} is piecewise constant with respect to the parameter $\lambda$.

We begin with introducing the notion of marginal functions, which plays an important role in optimization theory, see for example \cite{Teboulle:asympt2003}. Set
\begin{equation}\label{def:sigmaStar}
\sigma^*:=\min\{\|Ax-b\|_p: x\in\mathbb{R}^n\},
 \end{equation}
 where $p\in\{1, 2\}$. By Lemma \ref{lema:L0constSolu}, we have $\sigma^*\ge 0$ is well defined.   Let $H: [\sigma^*, +\infty)\rightarrow \mathbb{R}$ defined at $\sigma\ge \sigma^*$ as
\begin{equation}\label{def:H}
H(\sigma):=\min\{\|x\|_0: x\in G_\sigma\},
\end{equation}
where $G_\sigma$ is defined by \eqref{def:G}.
Clearly, for a fixed $\sigma\ge \sigma^*$, $H(\sigma)$ is the optimal value of problem \eqref{prob:constrain}. Thus $H$ is well defined due to  Theorem \ref{thm:existConst}.  The function $H$ is also called the marginal function of problem \eqref{prob:constrain}.
Similarly, we can define the marginal function of problem \eqref{prob:penalty}.  Given $\phi$ satisfying H\ref{hypo:1} and $p\in\{1, 2\}$, we define $F:(0, +\infty)\rightarrow\mathbb{R}$  at $\lambda>0$ as
\begin{equation}\label{def:F}
F(\lambda):=\min\{\|x\|_0+\lambda\phi(\|Ax-b\|_p): x\in\mathbb{R}^n\}.
\end{equation}
Clearly,  $F$ is well defined due to  Theorem \ref{thm:exist}.

In order to study the stability of the optimal solution sets of problems \eqref{prob:constrain} and \eqref{prob:penalty}, we  define $\widehat \Omega: [\sigma^*, +\infty)\rightarrow 2^{\mathbb{R}^n}$ at $\sigma\ge \sigma^*$ as
 the optimal solution set of the constrained  problem \eqref{prob:constrain}, that is,
\begin{equation}\label{def:omegaSharp}
\widehat \Omega(\sigma):=\arg\min\{\|x\|_0: x\in G_\sigma\},
\end{equation}
where $G_\sigma$ is defined by \eqref{def:G} and $2^{\mathbb{R}^n}$ collects all the subsets of $\mathbb{R}^n$. By Theorem \ref{thm:existConst}, $\widehat \Omega$ is well defined.
We also define  $\Omega:(0, +\infty)\rightarrow 2^{\mathbb{R}^n}$ at $\lambda>0$ as
\begin{equation}\label{def:Omega}
\Omega(\lambda):=\arg\min\{\|x\|_0+\lambda\phi(\|Ax-b\|_p): x\in\mathbb{R}^n\}.
\end{equation}
For a fixed $\lambda>0$, $\Omega(\lambda)$ is the optimal solution set of problem \eqref{prob:penalty}, therefore, is also well defined due to Theorem \ref{thm:exist}.  Then, our task in this section is establishing the properties of  functions $H$ and $F$ as well as properties of  mappings $\widehat \Omega$ and ${\Omega}$.
\subsection{Stability for problem \eqref{prob:constrain}}
This subsection is devoted to the stability for problem \eqref{prob:constrain}. We explore the behaviors of the marginal function $H$ and optimal solution set $\widehat \Omega$ with respect to $\sigma$.

We begin with a lemma which is crucial in our discussion.
\begin{lemma}\label{lemma:Lambdastrict}
Given $\phi$ satisfying H\ref{hypo:1} and $p\in\{1, 2\}$, let $L$ and $s_i, \rho_i, \Omega_i$ for $i\in\mathbb{N}_{L}^0$ be defined by Definition \ref{def:sRhoOmega}. For any $x\in\Omega_i$, $i \in\mathbb{N}_{L}^0$, let $\Lambda:=\mathrm{supp}(x)$. Then the following hold:
\begin{itemize}
\item [(i)] $A_\Lambda$ has full column rank.
\item [(ii)] If $\sharp\{z:\phi(z)=\rho_i\}=1$, then $x_\Lambda$ is the unique  minimizer of $\min\{\|A_\Lambda y-b\|_2: y\in\mathbb{R}^{s_i}\}$, that is
    for any $x_\Lambda\neq y\in\mathbb{R}^{s_k}$, $\|A_\Lambda y-b\|_2>\|A_\Lambda x_\Lambda-b\|_2$.
\end{itemize}
\end{lemma}
\begin{proof}
We first prove Item (i). We prove it by contradiction. If not, there exists $x'\in\mathbb{R}^n$ such that $\|x'\|_0<s_i$ and $x'_{\mathbb{N}_n\backslash\Lambda}=\mathbf{0}_{n-s_i}$ such that $Ax'=Ax$. Therefore, $\phi(\|Ax'-b\|_p)=\phi(\|Ax-b\|_p)=\rho_i$.  This contradicts the definition of $s_i$. Then we get Item (i).

We next prove Item (ii). Since $\sharp\{z:\phi(z)=\rho_i\}=1$ and $\phi$ is nondecreasing, it follows that  $\arg\min\{\phi(\|Ax-b\|_p):\|x\|_0\le s_{i-1}-1\}=\arg\min\{\|Ax-b\|_p:\|x\|_0\le s_{i-1}-1\}$. Thus, $x\in\arg\min\{\|A_\Lambda y-b\|_p: y\in\mathbb{R}^{s_i}\}$ for any $x\in\Omega_i$. The problem $\min\{\|A_\Lambda y-b\|_2: y\in\mathbb{R}^{s_i}\}$ is equivalent to the problem \begin{equation}\label{eq:14}
\min\{\|A_\Lambda y-b\|_2^2: y\in\mathbb{R}^{s_i}\}.
\end{equation} Then from Item (i),  the objective function of the convex optimization problem \eqref{eq:14} is strictly convex. Therefore, we obtain Item (ii).
\end{proof}

We next show an important proposition which reveals the relationship between $\Omega_i$ and the optimal solution set of problem \eqref{prob:constrain}.
Clearly, when $\sigma\ge \|b\|_p$, $\widehat\Omega(\sigma)=\Omega_L=\{\mathbf{0}_n\}$. Thus, we only discuss cases as $\sigma<\|b\|_p$.
\begin{proposition}\label{prop:partialEauvalence}
Given $\phi$ satisfying H\ref{hypo:1} and $p\in\{1, 2\}$, let $L$ and $s_i, \rho_i, \Omega_i$ for $i\in\mathbb{N}_{L}^0$ be defined by Definition \ref{def:sRhoOmega}.  Let $\widehat \Omega$, $G_\sigma$ and $\sigma^*$ be defined by \eqref{def:omegaSharp}, \eqref{def:G} and \eqref{def:sigmaStar} respectively. Suppose  $\phi$ is strictly increasing. Then, the following statements hold:
\begin{itemize}
\item [(i)]  For any $\sigma\in [\sigma^*, \|b\|_p)$, there must exist $k\in \mathbb{N}_{L-1}^0$ such that $\rho_k\le\phi(\sigma)<\rho_{k+1}$, and $\Omega_k$ is a subset of $\widehat \Omega(\sigma)$, that is
    \begin{equation}\label{eq:13}
    \Omega_k\subseteq \widehat \Omega(\sigma),
    \end{equation}
    moreover, $\|y\|_0=s_k$ for any $y\in\widehat \Omega(\sigma)$.
    \end{itemize}
\begin{itemize}
\item [(ii)] Particularly, if the $k$ in Item (i) satisfies $\rho_k=\phi(\sigma)$, then $\widehat \Omega(\sigma)$ is equal to $\Omega_k$, that is
    \begin{equation}\label{eq:10}
    \widehat \Omega(\sigma)=\Omega_k.
    \end{equation}
\item [(iii)] Further, if $p=2$ and the $k$ in Item (i) satisfies $\rho_k<\phi(\sigma)<\rho_{k+1}$, then $\Omega_k$ is a real subset of $\widehat \Omega(\sigma)$, that is
    \begin{equation}\label{eq:9}
    \Omega_k\subset \widehat \Omega(\sigma).
    \end{equation}
    \end{itemize}
\end{proposition}
\begin{proof}
We first prove Item (i). Since $\phi$ is strictly increasing,  $\rho_{L}=\phi(\|b\|_p)$ and $\rho_0=\phi(\sigma^*)$, there must exist $k\in \mathbb{N}_{L-1}^0$ such that $\rho_k\le\phi(\sigma)<\rho_{k+1}$ due to Item (iii) of Proposition \ref{prop:rhoSOmega}. We then prove $\Omega_k\subseteq \widehat \Omega(\sigma)$. Let $x\in\Omega_k$, we will show $x\in\widehat \Omega(\sigma)$.
Since $\rho_k=\phi(\|Ax-b\|_p)$, $\rho_k\le \phi(\sigma)$ and $\phi$ is strictly increasing, we have $x\in G_\sigma$.
In order to prove $x\in\widehat \Omega(\sigma)$, we need to show that for any $y\in G_\sigma$ there holds $\|x\|_0\le \|y\|_0$.  If not,  there exists $y\in G_\sigma$ such that $\|y\|_0\le \|x\|_0-1=s_k-1$. By the definition of $\rho_{k+1}$, we have $\rho_{k+1}\le \phi(\|Ay-b\|)\le \phi(\sigma)$, which contradicts the fact that $\rho_{k+1}>\phi(\sigma)$. Therefore, $x\in\Omega_k$ implies $x\in\widehat \Omega(\sigma)$. Moreover, since $\|x\|_0=s_k$ we have $H(\sigma)=s_k$, where $H$ is the marginal function of problem \eqref{prob:constrain}, defined by \eqref{def:H}. Therefore, we have $\|y\|_0 =H(\sigma)=s_k$ for any $y\in\widehat \Omega(\sigma)$. We then get Item (i).

Next, we try to prove Item (ii). By Item (i) of this proposition, we have $\|y\|_0=s_k$ for any $y\in\widehat \Omega(\sigma)$. In order to prove \eqref{eq:10}, we only need to show $\phi(\|Ay-b\|_p)=\rho_k$ for any $y\in\widehat \Omega(\sigma)$ due to Item (v) of Proposition \ref{prop:rhoSOmega}. We next show it by contradiction. If there exists $y^*\in\widehat \Omega(\sigma)$ such that $\phi(\|Ay^*-b\|_p)\neq\rho_k$, then
\begin{equation}\label{eq:12}
\phi(\|Ay^*-b\|_p)<\rho_k
\end{equation}
due to $y^*\in G_\sigma$, the strictly increasing property of  $\phi$ and $\rho_k=\phi(\sigma)$. Since $\|y^*\|_0=s_k\le s_{k-1}-1$, by the definition of $\rho_k$, we have $\rho_k\le \phi(\|Ay^*-b\|_p)$, contradicting \eqref{eq:12}. Thus we get Item (ii) of this proposition.

Finally, we prove Item (iii).
 We only need to show that there exists $y\in\widehat \Omega(\sigma)$ such that $y\not\in\Omega_k$. Let $x\in\Omega_k$, we have $\|x\|_0=s_k\ge 1$. Let $\Lambda:=\mathrm{supp}(x)$.
Since $\phi$ is continuous and $\rho_k=\phi(\|Ax-b\|_p)<\phi(\sigma)$, there exists $\delta_1>0$ such that as long as $\|x-x'\|_2\le \delta_1$, $\phi(\|Ax'-b\|_p)\le \phi(\sigma)$. Let $\delta_2:=\frac{1}{2}\min\{|x_i|: i\in\Lambda\}$. Then $\delta_2>0$. Set $\delta:=\min\{\delta_1, \delta_2\}$. Set $y\in S:=\{x': \|x'-x\|_2\le \delta, x_{\mathbb{N}_n\backslash\Lambda}'=\mathbf{0}_{n-s_k}, x'\neq x, x'\in\mathbb{R}^n \}$. Obviously, $S\neq \emptyset$,  $\|y\|_0=\|x\|_0=s_k$ and $\phi(\|Ay-b\|_p)\le \phi(\sigma)$. Since $\phi$ is strictly increasing, we have $y\in G_\sigma$. By Item (i) of this proposition, $y\in\widehat \Omega(\sigma)$ due to $\|y\|_0=s_k$ and $y\in G_\sigma$. We then show $y\not\in\Omega_k$. It amounts to showing $\phi(\|Ay-b\|_p)\neq \phi(\|Ax-b\|_p)$ by Item (v) of Proposition \ref{prop:rhoSOmega}. By Item (ii) of Lemma \ref{lemma:Lambdastrict}, we have $\|Ay-b\|_p>\|Ax-b\|_p$ since  $y$ and $x$ share the same support $\Lambda$. Then by the strictly increasing property of $\phi$ we get $y\not\in\Omega_k$. Then Item (iii) of this proposition follows.
\end{proof}

\begin{remark}\label{remark:same}
 One can easily check that Definition \ref{def:sRhoOmega} produces the same $L$ and $s_i, \Omega_i$ for $i\in\mathbb{N}_{L}^0$ for different choices of $\phi$ once they are all strictly increasing. Therefore, it is reasonable that $\widehat \Omega(\sigma)$ has connections with $\Omega_i$ for $i\in\mathbb{N}_{L}^0$, as shown in Proposition \ref{prop:partialEauvalence}.
\end{remark}

Now, we are ready to establish the main results of this subsection.
\begin{theorem}\label{thm:staConst}
Given $\phi:\mathbb{R}_+\rightarrow \mathbb{R}$ defined at $z\ge0$ as $\phi(z)=z$ and $p\in\{1, 2\}$, let $L$ and $s_i, \rho_i, \Omega_i$ for $i\in\mathbb{N}_{L}^0$ be defined by Definition \ref{def:sRhoOmega}. Let  $H$ and $\widehat \Omega$ be defined by   \eqref{def:H} and \eqref{def:omegaSharp} respectively.  Then the following statements hold:
\begin{itemize}
\item [(i)]
$$
H(\sigma)=\begin{cases}
s_i, &\mathrm{~~if~~}\sigma\in[\rho_i, \rho_{i+1}), i\in\mathbb{N}_{L-1}^0,\\
0, &\mathrm{~~if~~}\sigma\ge \|b\|_p.
\end{cases}
$$
\item [(ii)]
$$
\begin{cases}
\widehat\Omega(\sigma)=\Omega_i, &\mathrm{~~if~~}\sigma=\rho_i, i\in\mathbb{N}_{L-1}^0,\\
\Omega_i\subseteq \widehat\Omega(\sigma),&\mathrm{~~if~~}\sigma\in(\rho_i, \rho_{i+1}), i\in\mathbb{N}_{L-1}^0,\\
\widehat\Omega(\sigma)=\{\mathbf{0}_n\},&\mathrm{~~if~~}\sigma\ge\|b\|_p.
\end{cases}
$$
In particular, if $p=2$, we have $\Omega_i\subset\widehat\Omega(\sigma)$ for $\sigma\in(\rho_i, \rho_{i+1})$, $i\in\mathbb{N}_{L-1}^0$.
\item [(iii)] $H$ is a piecewise constant function and  has at most $\rank(A)+1$ values. Further, $H$ is lower semicontinuous, right-continuous and nonincreasing.
\item [(iv)] $\widehat \Omega(\sigma_1)\subseteq\widehat \Omega(\sigma_2)$ for any $\sigma_1, \sigma_2\in[\rho_i, \rho_{i+1})$ satisfying $\sigma_1\le \sigma_2$, $i\in\mathbb{N}_{L-1}^0$.
\end{itemize}
\end{theorem}
\begin{proof}
Items (i) and (ii) follow from  Proposition \ref{prop:partialEauvalence}. Then Item (iii) of this theorem follows immediately. By Item (i) and $G_{\sigma_1}\subseteq G_{\sigma_2}$ for any $\sigma_1, \sigma_2\in[\rho_i, \rho_{i+1})$ satisfying $\sigma_1\le \sigma_2$, we obtain Item (iv).
\end{proof}

\subsection{Stability for problem \eqref{prob:penalty}}\label{sec:staP}
This subsection is for the stability of problem \eqref{prob:penalty}. We shall study the behaviors of
the marginal function $F$ and optimal solution set $\Omega$ with respect to $\lambda$.

 According to Theorem \ref{thm:exist}, $F(\lambda)=\min\{f_i(\lambda), i \in\mathbb{N}_{L}^0\}$ with $f_i$ defined by \eqref{def:fi}. Obviously,  each $f_i$ for $i \in\mathbb{N}_{L}^0$ is a line with slop $\rho_i$ and intercept $s_i$. Therefore, by Items (ii) and (iii) of Proposition \ref{prop:rhoSOmega}, it is easy to deduce that $F$ is continuous, piecewise linear and nondecreasing. Next we propose an iteration procedure to find the minimal value of $f_i$ for $i \in\mathbb{N}_{L}^0$, i.e., the marginal function $F(\lambda)$, for $\lambda\in(0,+\infty)$.

\begin{definition}\label{def:lambda_i}
Given $\phi$ satisfying H\ref{hypo:1} and $p\in\{1, 2\}$, let $L$ and $s_i, \rho_i, \Omega_i$ for $i\in\mathbb{N}_{L}^0$ be defined by Definition \ref{def:sRhoOmega}. Then the integer $K$, $\{t_i\in \mathbb{N}_{L}^0:i\in\mathbb{N}_{K}^0\}, \{\lambda_i>0: i \in\mathbb{N}_{K}^0\}, \{\Lambda_i\subseteq \mathbb{N}_{L}^0: i=-1,0,\dots, K-1\}$ are defined by the following iteration
\begin{eqnarray*}
\mathrm{set~~}& i=0, &\Lambda_{-1}:=\{0\}, \\
\mathrm{while~~}&L\not\in\Lambda_{i-1} &\\
&t_{i}:=&\max \Lambda_{i-1},\\
&\lambda_i:=&\max\{\frac{s_{t_i}-s_j}{\rho_j-\rho_{t_i}}: j=t_i+1,t_i+2,\dots, L\},\\
&\Lambda_i:=&\arg\max\{\frac{s_{t_i}-s_j}{\rho_j-\rho_{t_i}}: j=t_i+1,t_i+2,\dots, L\},\\
&i=&i+1,\\
\mathrm{end~~}&&\\
&K:=i,& t_K:=L, \lambda_K:=0.\\
\end{eqnarray*}
\end{definition}

We present an example to illustrate the iterative procedure in Definition \ref{def:lambda_i}.
\begin{example}\label{exam:3.6}
In this example, we set $L=4$. By Proposition \eqref{prop:rhoSOmega}, we set  $s_0=4, s_1=3, s_2=2, s_3=1, s_4=0$ and $\rho_0=0, \rho_1=0.25, \rho_2=0.5, \rho_3=0.5+1/1.5, \rho_4=1.5$. Then $f_i$ for $i\in\mathbb{N}_4^0$ can be obtained by equation \ref{def:fi}. 
We depict the plots of $f_i$ for $i\in\mathbb{N}_4^0$ in Figure \ref{fig:ex1}.
We observe from Figure \ref{fig:ex1} that, when $\lambda>\lambda_0$, $F(\lambda)=f_{t_0}(\lambda)=f_0(\lambda)$. Further,  the first  and largest critical parameter $\lambda_0$ can be obtained by the largest intersection point of $f_0$ and $f_i$ for all $i\in\mathbb{N}_4$. Here, we have that $\Lambda_0=\{1, 2\}$, which collects all the indexes that $f_i(\lambda_0)=f_0(\lambda_0)$ except $t_0=0$. In order to search for the second critical parameter value $\lambda_1$, we first find that on the second interval $(\lambda_1, \lambda_0)$, $F(\lambda)=f_{t_1}(\lambda)$, where $t_1=\max\Lambda_0=2$. Thus, $\lambda_1$ can be obtained  by the largest intersection point of $f_2$ and $f_j$ for $j=3, 4$. Then $\Lambda_1=\{L\}=\{4\}$. Thus, in this case, $K=2$, $t_2=4$. On interval $(0, \lambda_1)$, we have $F(\lambda)=f_{t_2}(\lambda)=f_4(\lambda)$.
In this example, $F(\lambda)\neq f_3(\lambda)$ for any $\lambda>0$.

\begin{figure}\label{fig:ex1}
\centering

\scalebox{0.8}{\includegraphics{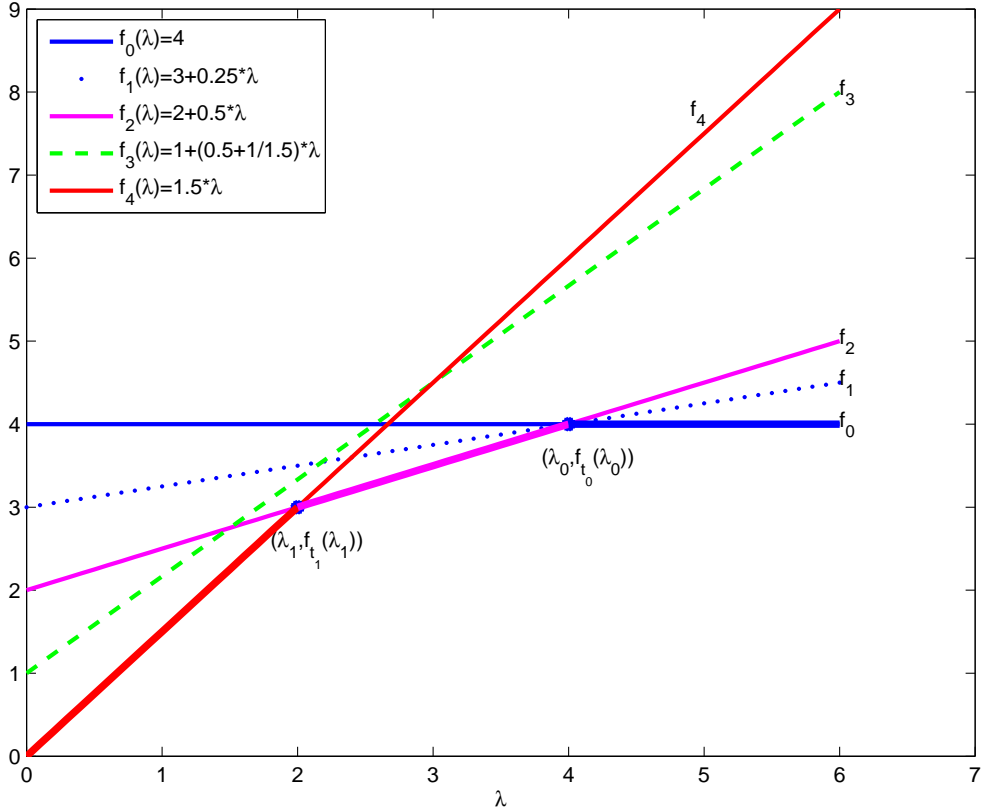}}

\caption{Plots of $f_i$ for $i\in\mathbb{N}_4^0$ in Example \ref{exam:3.6}. The line in bold represents $F(\lambda):=\min\{f_i(\lambda): i\in\mathbb{N}_4^0\}$. }
\end{figure}
\end{example}

The following proposition provides some basic properties of $K$, $\{\Lambda_i\subseteq \mathbb{N}_{L}^0: i=-1,0,\dots, K-1\}$ and $t_i, \lambda_i$ for $i\in\mathbb{N}_{K}^0$ by Definition \ref{def:lambda_i}.
\begin{proposition}\label{prop:lambda_i}
Given $\phi$ satisfying H\ref{hypo:1} and $p\in\{1, 2\}$, let $L$ and $s_i, \rho_i, \Omega_i$ for $i\in\mathbb{N}_{L}^0$ be defined by Definition \ref{def:sRhoOmega}. Let  $K$, $\{\Lambda_i\subseteq \mathbb{N}_{L}^0: i=-1,0,\dots, K-1\}$ and $t_i, \lambda_i$ for $i\in\mathbb{N}_{K}^0$ be  defined by  Definition \ref{def:lambda_i}. Then the following statements hold:
\begin{itemize}
\item [(i)] $0\le K\le L$, in particular, if $L\ge 1$ then $K\ge 1$.
\item [(ii)] $0=t_0<t_1<\dots<t_K=L$.
\item [(iii)] $\lambda_0>\lambda_1>\dots>\lambda_K=0$.
\item [(iv)] $\Lambda_i\neq \emptyset$ and $\Lambda_i\cap\Lambda_j=\emptyset$, for all $i\neq j$, $i, j=-1, 0,\dots,  K-1$.
\end{itemize}
\end{proposition}
The proof is outlined in Appendix 8.1.

The main results of this subsection are presented in the following theorem.
\begin{theorem}\label{thm:stability}
Given $\phi$ satisfying H\ref{hypo:1} and $p\in\{1, 2\}$, let $L$ and $s_i, \rho_i, \Omega_i$ for $i\in\mathbb{N}_{L}^0$ be defined by Definition \ref{def:sRhoOmega}. Let  $K$, $\{\Lambda_i\subseteq \mathbb{N}_{L}^0: i=-1,0,\dots, K-1\}$ and $t_i, \lambda_i$ for $i\in\mathbb{N}_{K}^0$ be  defined by  Definition \ref{def:lambda_i}. Let $f_i$ for $i \in\mathbb{N}_{L}^0$, $F$ and $\Omega$ be defined by \eqref{def:fi}, \eqref{def:F} and \eqref{def:Omega} respectively. By Remark \ref{remark:L=0}, if $L=0$, then $F(\lambda)=\rho_0\lambda$ and $\Omega(\lambda)=\{\mathbf{0}_n\}$. If $L\ge 1$, then the following statements hold:
\begin{itemize}
\item [(i)] $$F(\lambda)=\begin{cases}
f_L(\lambda), &\mathrm{~~if~~}\lambda\in(0, \lambda_{K-1}],\\
f_{t_i},& \mathrm{~~if~~}\lambda\in (\lambda_i, \lambda_{i-1}], i\in\mathbb{N}_{K-1},\\
f_0(\lambda),& \mathrm{~~if~~}\lambda\in(\lambda_0, +\infty).
\end{cases}$$
\item [(ii)] $$
\Omega(\lambda)=\begin{cases}
\Omega_L=\{\mathbf{0}_n\},& \mathrm{~~if~~}\lambda\in(0, \lambda_{K-1}),\\
\Omega_{t_i}, & \mathrm{~~if~~}\lambda\in (\lambda_i, \lambda_{i-1}), i\in\mathbb{N}_{K-1},\\
\Omega_0,& \mathrm{~~if~~}\lambda\in(\lambda_0, +\infty),\\
\bigcup_{k\in\Lambda_i}\Omega_k\cup\Omega_{t_i}, & \mathrm{~~if~~} \lambda=\lambda_i, i\in\mathbb{N}_{K-1}^0.
\end{cases}
$$
\item [(iii)] $F$ is continuous, piecewise linear and nondecreasing, particularly,
$F$ is strictly increasing in $(0, \lambda_0)$.
\end{itemize}
\end{theorem}
To improve readability, the proof of Theorem \ref{thm:stability} and two auxiliary lemma and proposition are given in Appendix 8.2.

A direct consequence of Theorem \ref{thm:stability} is stated below.
\begin{corollary}\label{crol:stability}
Under the assumptions of Theorem \ref{thm:stability},  the following statements hold:
\begin{itemize}
\item [(i)] If $\lambda', \lambda''\in (\lambda_0, +\infty)$ or $\lambda', \lambda''\in (\lambda_i, \lambda_{i-1})$, $i\in\mathbb{N}_K$, then $\|x\|_0=\|y\|_0$ and $\phi(\|Ax-b\|_p)=\phi(\|Ay-b\|_p)$ hold for any $x\in\Omega(\lambda')$ and any $y\in\Omega(\lambda'')$.
\item [(ii)] If $\lambda'<\lambda''$, then $\|x\|_0\le \|y\|_0$ and $\phi(\|Ax-b\|_p)\ge \phi(\|Ay-b\|_p)$ hold for any $x\in\Omega(\lambda')$ and any $y\in\Omega(\lambda'')$.
\end{itemize}
\end{corollary}
\begin{proof}
Item (i) can be obtained by Item (ii) of Theorem \ref{thm:stability} as well as Item (v) of Proposition \ref{prop:rhoSOmega}. Similarly, Item (ii) follows from Item (ii) of  Theorem \ref{thm:stability}, Items (ii), (iii), (v) of Proposition \ref{prop:rhoSOmega} as well as Item (ii) of Proposition \ref{prop:lambda_i}.
\end{proof}

From Theorem \ref{thm:stability}, the optimal value of problem \eqref{prob:penalty} changes piecewise linearly while the optimal solution set of problem \eqref{prob:penalty} changes piecewise constantly as the parameter $\lambda$ varies. In addition, by Corollary \ref{crol:stability}, the optimal values of both the first and second terms of \eqref{prob:penalty}  are piecewise constant with respect to changes in the parameter $\lambda$.

\begin{remark}
We observe from Theorem \ref{thm:staConst} and \ref{thm:stability} that, when $\sigma\ge \|b\|_p$, problem \eqref{prob:constrain} share the same optimal solution set $\{\mathbf{0}_n\}$ as problem \eqref{prob:penalty} if $\lambda$ is chosen to be small enough.  Thus, in the remaining part of this paper, we only focus on the nontrivial case when $\sigma \in[\sigma^*, \|b\|_p)$, where $\sigma^*$ is defined by \eqref{def:sigmaStar}.
\end{remark}
\section{Relationship between the the constrained problem and a special penalty problem}\label{sec:relation}
In this section we explore the relationship between the constrained problem \eqref{prob:constrain} with  $p\in\{1,2\}$ and a corresponding special penalty problem. In the case of $p=2$ we shall study the relationship between problem \eqref{prob:l2-constrain} and the least square penalty problem \eqref{prob:l2-unconstrain}. In the case of $p=1$, we shall discuss the connections between problem \eqref{prob:l1-constrain} and the following $\ell_1$-penalty problem
\begin{equation}\label{prob:l1-unconstrain}
\min\{\|x\|_0+{\lambda}\|Ax-b\|_1: x\in\mathbb{R}^n\}.
\end{equation}
In fact, both problems \eqref{prob:l2-unconstrain} and \eqref{prob:l1-unconstrain} can be cast into the general penalty problem \eqref{prob:penalty} with $\phi$ defined at $z\ge0$, for $p\in\{1, 2\}$, as
\begin{equation}\label{eq:spePhi}
\phi(z):=\frac{1}{p}z^p.
\end{equation}


We exhibit the main results of this section in the following theorem.
\begin{theorem}\label{thm:leasEquv}
For $\phi$ defined by \eqref{eq:spePhi} and $p\in\{1,2\}$,  let $L$ and $s_i, \rho_i, \Omega_i$ for $i\in\mathbb{N}_{L}^0$ be defined by Definition \ref{def:sRhoOmega}. Let  $K$, $\{\Lambda_i\subseteq \mathbb{N}_{L}^0: i=-1,0,\dots, K-1\}$ and $t_i, \lambda_i$ for $i\in\mathbb{N}_{K}^0$ be  defined by  Definition \ref{def:lambda_i}.  Let  $\lambda_{-1}=+\infty$ and $T:=\{t_i: i\in\mathbb{N}_{K}^0\}$. Let $\sigma^*$, $\Omega$ and $\widehat \Omega$ be defined by \eqref{def:sigmaStar}, \eqref{def:Omega} and \eqref{def:omegaSharp} respectively. For $\sigma\in[\sigma^*, \|b\|_p)$, let $k\in\mathbb{N}_{L-1}^0$ be the integer such that $\rho_k\le \phi(\sigma)<\rho_{k+1}$.
\begin{itemize}
 \item [(i)] If $k\in T$ and suppose $k=t_j$,  then
     $$
     \begin{cases}
     \Omega(\lambda)\subseteq\widehat\Omega(\sigma), &\mathrm{~~if~~}\lambda\in(\lambda_j, \lambda_{j-1}),\\
     \Omega(\lambda)\cap\widehat\Omega(\sigma)=\Omega_{k}, &\mathrm{~~if ~~}\lambda=\lambda_{j} \mathrm{~~or~~} \lambda=\lambda_{j-1},\\
     \Omega(\lambda)\cap\widehat\Omega(\sigma)=\emptyset,&\mathrm{~~else.}
     \end{cases}
     $$
     Further, if $\phi(\sigma)=\rho_k$, we have $\Omega(\lambda)=\widehat\Omega(\sigma)$ for $\lambda\in(\lambda_j, \lambda_{j-1})$ and $\widehat\Omega(\sigma)\subset \Omega(\lambda)$ for $\lambda=\lambda_j, \lambda_{j-1}$.

\item [(ii)]  If  $k\in\bigcup_{i=-1}^{K-1}\Lambda_i\backslash T$, and suppose $k\in\Lambda_j$, then
    $$
    \begin{cases}
    \Omega(\lambda)\cap \widehat\Omega(\sigma)=\Omega_k, &\mathrm{~~if~~}\lambda=\lambda_j\\
    \Omega(\lambda)\cap \widehat\Omega(\sigma)=\emptyset, &\mathrm{~~else}.
    \end{cases}
    $$

\item [(iii)] If $k\not\in \bigcup_{i=-1}^{K-1}\Lambda_i$, then $\Omega(\lambda)\cap\widehat \Omega(\sigma)=\emptyset$ for any $\lambda>0$.
\end{itemize}

\end{theorem}
\begin{proof}
We first prove Item (i).
From Item (ii) of Theorem \ref{thm:stability},  we have $\Omega_{k}=\Omega(\lambda)$ for  $\lambda\in(\lambda_j, \lambda_{j-1})$. By Item (i) of Proposition \ref{prop:partialEauvalence}, we get that $\Omega(\lambda)\subseteq \widehat\Omega(\sigma)$ as $\lambda\in(\lambda_j, \lambda_{j-1})$. We also obtain that if $\phi(\sigma)=\rho_k$, $\Omega(\lambda)= \widehat\Omega(\sigma)=\Omega_k$ for $\lambda\in(\lambda_j, \lambda_{j-1})$ from  Item (ii) of Proposition \ref{prop:partialEauvalence}. When $\lambda=\lambda_j$ or $\lambda=\lambda_{j-1}$, by item (ii) of Theorem \ref{thm:stability} and $k=t_j\in\Lambda_{j-1}$, we have $\Omega_k\subset \Omega(\lambda)$ and $\|x\|_0\neq s_k$ for any $x\in \Omega(\lambda)\backslash \Omega_k$. Since $\Omega_k\subseteq\widehat\Omega(\sigma)$ and $\|x\|_0=s_k$ for all $x\in\widehat\Omega(\sigma)$, we get $\Omega(\lambda)\cap \widehat\Omega(\sigma)=\Omega_k$ for $\lambda=\lambda_j$ or $\lambda=\lambda_{j-1}$. Clearly, if $\phi(\sigma)=\rho_k$, then $\Omega_k=\widehat\Omega(\sigma)\subset\Omega(\lambda)$ for $\lambda=\lambda_j$ or $\lambda=\lambda_{j-1}$. When $\lambda\not\in[\lambda_j, \lambda_{j-1}]$, $\|x\|_0\neq s_k$ for $x\in\Omega(\lambda)$ and thus $\Omega(\lambda)\cap\widehat\Omega(\sigma)=\emptyset$.

We then prove Item (ii). By Item (i) of Proposition \ref{prop:partialEauvalence},  $\|x\|_0=s_k$ for any $x\in\widehat \Omega(\sigma)$. From Item (ii) of Theorem \ref{thm:stability}, we have $\Omega_k\subset \Omega(\lambda_j)$ due to $k\in\Lambda_j$. Then we get Item (ii) by the fact that $\|x\|_0\neq s_k$ for any $x\in\Omega(\lambda_j)\backslash \Omega_k$ and for any $x\in \Omega(\lambda)$ as $\lambda\neq \lambda_j$.

We finally prove Item (iii). Similarly,  we have $\|x\|_0=s_k$ for any $x\in\widehat \Omega(\sigma)$. However, by Theorem \ref{thm:stability} $\|x\|_0\neq s_k$ for any $x\in\Omega(\lambda)$ and $\lambda>0$ since $k\not\in\bigcup_{i=-1}^{K-1}\Lambda_i$. Thus we get Item (iii).
\end{proof}
\begin{remark}\label{remark:p=2}
In the case of $p=2$, by Proposition \ref{prop:partialEauvalence}, if $\rho_k<\phi(\sigma)<\rho_{k+1}$ and $k\in T$, then we have $\Omega(\lambda)\subset\widehat\Omega(\sigma)$ when $\lambda\in(\lambda_j, \lambda_{j-1})$.
\end{remark}

\begin{remark}
If $K=L$, then we have $k\in T$. In such case, Item (i) of Theorem \ref{thm:leasEquv} holds clearly.
\end{remark}

According to Item (iii) of Theorem \ref{thm:leasEquv}, in general it is possible that problem \eqref{prob:l2-unconstrain} has {\bf no} common optimal solutions with problem \eqref{prob:l2-constrain} for any $\lambda>0$. This means that in general problem \eqref{prob:l2-unconstrain} may not be an exact penalty method for problem \eqref{prob:l2-constrain}. However, in the noiseless case, we show in the next corollary that problem \eqref{prob:l2-unconstrain} has the same optimal solution set as problem \eqref{prob:lineq} if $\lambda$ is large enough.
\begin{corollary}\label{coro:least}
Let $\phi$ be defined by  \eqref{eq:spePhi} and $p\in\{1,2\}$. Assume the feasible region of problem \eqref{prob:lineq} is
nonempty. Then, there exists $\lambda^*>0$ such that problems \eqref{prob:lineq}
and \eqref{prob:penalty}  with $\lambda>\lambda^*$ have the same optimal solution set.
\end{corollary}
\begin{proof}
Since the feasible region of problem \eqref{prob:lineq} is nonempty, $\rho_0=0$ by Definition \ref{def:sRhoOmega}.
Clearly, problem \eqref{prob:lineq} is exactly  problem \eqref{prob:constrain} with $\sigma=0$. Then, $\phi(\sigma)=\rho_0=0$ in this case. Let $K$ and $t_i$, $\lambda_i$ for $i\in\mathbb{N}_K^0$ be defined by definition \ref{def:lambda_i}. From Item (ii) of Proposition \ref{prop:lambda_i}, $t_0=0$ and thus $0\in T:=\{t_i, i\in\mathbb{N}_K^0\}$. According to Item (i) of Theorem \ref{thm:leasEquv},  problems \eqref{prob:lineq}
and \eqref{prob:penalty}   have the same optimal solution set when $\lambda\in(\lambda_0, +\infty)$. Set $\lambda^*=\lambda_0$, we then obtain the desired results.
\end{proof}

\section {Exact penalty problems for constrained $\ell_0$ regularization}\label{sec:equal}
In this section, we first establish exact penalty conditions under which problem \eqref{prob:penalty} have the same optimal solution set as problem \eqref{prob:constrain}, provided that $\lambda>\lambda^*$ for some $\lambda^*>0$. Based on the conditions developed, we propose several exact penalty formulations for problem \eqref{prob:constrain}.

 According to the definition of $\rho_0$ and $\Omega_0$ in Definition \ref{def:sRhoOmega}, the penalty term $\phi(\|A\cdot-b\|_p)$ attains its minimal value at any $x\in\Omega_0$. We also know that $\phi(\|A\cdot-b\|_2)$ attains its minimal value at any $x\in\Omega(\lambda)$ when $\lambda>\lambda_0$ by Item (ii) of Theorem \ref{thm:stability}, where $\lambda_0$ is defined by Definition \ref{def:lambda_i}. Since $\phi$ is nondecreasing, we have $\phi(\|Ax-b\|_p)\le \phi(\sigma)$ for any $x\in G_\sigma$. Then we obtain the following theorem.

\begin{theorem}\label{thm:equivalence}
  Let $\phi$ satisfy H\ref{hypo:1} and $p\in\{1, 2\}$. Suppose the feasible region of problem \eqref{prob:constrain} is nonempty, that is $G_\sigma\neq \emptyset$, where $G_\sigma$ is defined by \eqref{def:G}.  If
\begin{equation}\label{eq:eqiv}
\arg\min\{\phi(\|Ax-b\|_p): x\in\mathbb{R}^n\}=G_\sigma,
\end{equation}
then there exists $\lambda^*>0$ such that problems \eqref{prob:constrain} and \eqref{prob:penalty} with $\lambda>\lambda^*$ have the same optimal solution set.

\end{theorem}
\begin{proof}
Since \eqref{eq:eqiv} holds, by the definition of $\rho_0$ and $\Omega_0$ in Definition \ref{def:sRhoOmega}, $\Omega_0=\arg\min\{\|x\|_0: x\in G_\sigma\}$. Thus $\Omega_0$ is the optimal solution set of problem \eqref{prob:constrain}. Let $\lambda^*=\lambda_0$, where $\lambda_0$ is defined in Definition \ref{def:lambda_i}. Then, from Item (ii) of Theorem \ref{thm:stability}, $\Omega_0$ is also the optimal solution set of problem \eqref{prob:penalty} with $\lambda>\lambda^*$. Then we have  that problems \eqref{prob:constrain} and \eqref{prob:penalty} with $\lambda>\lambda^*$ have the same optimal solution set.
\end{proof}

In general, it is not easy to verify condition \eqref{eq:eqiv}. The next corollary states a simple and useful condition which also ensures the exact penalty property once $\lambda$ is large enough.
\begin{corollary}\label{crol:equv}
Let $\phi$ satisfy H\ref{hypo:1} and $p\in\{1, 2\}$. Assume that the feasible region of problem \eqref{prob:constrain} is nonempty.
If $\phi$ satisfies
\begin{equation}\label{eq:phiSpeci}
\phi(z)=0, ~z\le \sigma \mathrm{~~and~~} \phi(z)>0, ~ z>\sigma,
\end{equation}
then  there exists $\lambda^*>0$ such that problems \eqref{prob:constrain} and \eqref{prob:penalty} with $\lambda>\lambda^*$ have the same optimal solution set.
\end{corollary}
\begin{proof}
The desired results follows from the fact that condition \eqref{eq:phiSpeci} implies condition \eqref{eq:eqiv}.
\end{proof}

 Next we present several examples of $\phi$ satisfying condition \eqref{eq:phiSpeci}. For $z\in\mathbb{R}$, we denote by $z_+$ the positive part of $z$, that is,
 $$
z_+:=\begin{cases}
z,& \mathrm{~~if~~} z>0,\\
0,&\mathrm{~~else}.
\end{cases}$$
\begin{example}
Given $p\in\{1,2\}$. Let scalar function $\phi_1:\mathbb{R}_+\rightarrow\mathbb{R}$ be defined at $z\ge0$ as
$\phi_1(z):=(z^p-\sigma^p)_+$.
Then $\phi_1$ satisfies H1 and condition \eqref{eq:phiSpeci}. With the choice of $\phi=\phi_1$, the penalty term in problem \eqref{prob:penalty} becomes
\begin{equation}\label{eq:phi1squar}
\phi_1(\|Ax-b\|_p)=(\|Ax-b\|_p^p-\sigma^p)_+.
\end{equation}
\end{example}

The penalty term \eqref{eq:phi1squar} with $p=2$ is well exploited in penalty method for quadratically constrained $\ell_q$ minimization with $0<q<1$ in \cite{Chen-Lu:penalty2014}. One can easily check that the penalty term \eqref{eq:phi1squar} is convex but non-differentiable for $p\in\{1,2\}$.

\begin{example}
Given $p\in\{1,2\}$. Let $\phi_2:\mathbb{R}_+\rightarrow\mathbb{R}$ be defined at $z\ge0$ as
\begin{equation}\label{def:phi2}
\phi_2(z):=\frac{1}{2}(z-\sigma)_+^2.
\end{equation}
Then $\phi_2$ satisfies H1 and condition \eqref{eq:phiSpeci}. With the choice of $\phi=\phi_2$, the penalty term in problem \eqref{prob:penalty} becomes
\begin{equation}\label{eq:phi2squar}
\phi_2(\|Ax-b\|_p)=\frac{1}{2}(\|Ax-b\|_p-\sigma)_+^2.
\end{equation}
\end{example}

By choosing $\phi=\phi_2$, for $p=2$ problem \eqref{prob:penalty} reads as
\begin{equation}\label{prob:l2-env}
\min\{\|x\|_0+\frac{\lambda}{2}(\|Ax-b\|_2-\sigma)_+^2: x\in\mathbb{R}^n\}.
\end{equation}
According to Corollary \ref{crol:equv},  the following proposition concerning the exact penalization of problem \eqref{prob:l2-env} holds.

\begin{proposition}\label{prop:env}
Suppose that the feasible region of problem \eqref{prob:l2-constrain} is nonempty,  then the penalty problem \eqref{prob:l2-env} shares the same optimal solution set with the constrained problem \eqref{prob:l2-constrain}, provided that $\lambda>\lambda^*$ for some $\lambda^*>0$.
\end{proposition}

In particular, when $\sigma=0$, problem \eqref{prob:l2-env} reduces to the least square penalty problem \eqref{prob:l2-unconstrain}. From Proposition \ref{prop:env}, we deduce that \eqref{prob:l2-unconstrain} has the same optimal solution set as the linear constrained problem \eqref{prob:lineq} when $\lambda$ is sufficiently large. This result coincides with that obtained in Corollary \ref{coro:least}.


In addition to the advantage of  exact penalization, we prove in the next proposition that the penalty term in \eqref{prob:l2-env} is differentiable with a Lipschitz continuous gradient. For simplicity, we define $\Phi:\mathbb{R}^n\rightarrow\mathbb{R}$ at $x\in\mathbb{R}^n$ as
\begin{equation}\label{eq:defPhi}
\Phi(x):=\phi_2(\|Ax-b\|_2)=\frac{1}{2}(\|Ax-b\|_2-\sigma)_+^2.
\end{equation}
We denote by $\|B\|_2$  the largest singular value of $B\in\mathbb{R}^{d\times r}$. We also need to recall the notion of firmly nonexpansive. Operator $J:\mathbb{R}^d\rightarrow\mathbb{R}^r$ is called firmly nonexpansive (resp., nonexpansive) if for all $x, y\in\mathbb{R}^d$
\begin{equation}
\|Jx-Jy\|_2^2\le \langle Jx-Jy, x-y\rangle~~ (\mathrm{resp.,} \|Jx-Jy\|_2\le \|x-y\|_2).
\end{equation}
Obviously, a firmly nonexpansive operator is nonexpansive. Now, we are ready to show $\Phi$ is differentiable with a Lipshitz continuous gradient.

\begin{proposition}\label{prop:diff}
Let $\Phi:\mathbb{R}^n\rightarrow\mathbb{R}$ be defined by \eqref{eq:defPhi}. Then function $\Phi$ is differentiable with a Lipschitz continuous gradient. In particular, for $x\in\mathbb{R}^n$
\begin{itemize}
\item [(i)] $\nabla\Phi(x)=\begin{cases}
      \bf{0}_n,&\mathrm{~~if ~~}\|Ax-b\|_2\le \sigma,\\
    (1-\frac{\sigma}{\|Ax-b\|_2})A^\top (Ax-b), &\mathrm{~~else,}
    \end{cases}$
\item [(ii)] $\nabla\Phi$ is Lipschitz continuous with the Lipschitz constant $\|A\|_2^2$.
\end{itemize}
\end{proposition}
\begin{proof}
 We  define $\varphi:\mathbb{R}^m\rightarrow\mathbb{R}$, at $y\in\mathbb{R}^m$ as
 \begin{equation}\label{def:varphi}
 \varphi(y)=\frac{1}{2}(\|y\|_2-\sigma)_+^2.
 \end{equation}
 Obviously, $\Phi(x)=\varphi(Ax-b)$ for any $x\in\mathbb{R}^n$. One can  check that for any $y\in\mathbb{R}^m$,
 \begin{equation}\label{eq:8}
 \nabla\varphi(y)=\begin{cases}
      \bf{0}_n,&\mathrm{~~if ~~}\|y\|_2\le \sigma,\\
    (1-\frac{\sigma}{\|y\|_2})y, &\mathrm{~~else.}
    \end{cases}
 \end{equation}
Thus  $\nabla\Phi=A^\top\nabla\varphi(Ax-b)$ and Item (i) follows.

We next prove Item (ii). We first show $\nabla\varphi$ is nonexpansive. Let $P:\mathbb{R}^n\rightarrow \mathbb{R}^n$ be the projection operator onto the set $\{y: \|y\|_2\le \sigma\}$.  we have $\nabla\varphi=I-P$. Since the set $\{y: \|y\|_2\le \sigma\}$ is nonempty closed convex, operator $I-P$ is firmly nonexpansive, by Proposition 12.27 in \cite{Bauschke-Combettes:11}. Thus, $\nabla\varphi$ is nonexpansive. For any $x, z\in\mathbb{R}^n$, we have that
 \begin{eqnarray*}
 \|\nabla\Phi(x)-\nabla\Phi(z)\|_2&=&\|A^\top(\nabla\varphi(Ax-b)-\nabla\varphi(Az-b))\|_2\\
 &\le&\|A\|_2\|\nabla\varphi(Ax-b)-\nabla\varphi(Az-b)\|_2\\
 &\le&\|A\|_2\|Ax-Az\|_2\\
 &\le&\|A\|_2^2\|x-z\|_2.
 \end{eqnarray*}
 The above inequality leads to Item (ii).
\end{proof}

\begin{remark}
As mentioned in the introduction, the differentiability of the penalty term has important effect on the algorithmic design for problem \eqref{prob:penalty}. According to Proposition \ref{prop:diff}, by choosing $\phi=\phi_2$, the penalty term is differentiable with a Lipschitz continuous gradient in problem \eqref{prob:penalty} for $p=2$. As a result, we are able to develop efficient numerical algorithms with theoretical convergence guarantee for solving this problem, e.g., proximal-gradient type algorithms. 

\end{remark}
\section{Cardinality of optimal solution sets of problems \eqref{prob:constrain} and \eqref{prob:penalty}}\label{sec:minimizer}
In this section, we investigate the cardinality of optimal solution sets of problems \eqref{prob:constrain} and \eqref{prob:penalty}. We also discuss the strictness of their optimal solutions as a byproduct. As it is shown in Section \ref{sec:sta}, optimal solution sets of these two problems are closely related to $\Omega_k$, for $k\in\mathbb{N}_L^0$, defined by Definition \ref{def:sRhoOmega}. Therefore, we shall first consider the cardinality of $\Omega_k$. The cases of $p=1$ and $p=2$ are discussed in separate subsections, since the analysis and results are different for these two cases.

\subsection{Cardinality of $\Omega_k$ when $p=2$}
For the case of $p=2$, we have the following proposition concerning the cardinality of $\Omega_k$.
\begin{proposition}\label{prop:cardiP2}
For $\phi$ satisfying H\ref{hypo:1} and $p=2$,  let $L$ and $s_i, \rho_i, \Omega_i$ for $i\in\mathbb{N}_{L}^0$ be defined by Definition \ref{def:sRhoOmega}. For $k\in \mathbb{N}_{L}^0$, if $\sharp\{z:\phi(z)=\rho_k\}=1$,
 then we have
\begin{equation}\label{eq:15}
\sharp \Omega_k\le C_n^{s_k}.
\end{equation}

\end{proposition}
\begin{proof}
By the definition of $\Omega_k$, we have $\|x\|_0=s_k$ for any $x\in\Omega_k$. Let $\mathcal{O}:=\{\Lambda: \Lambda\subseteq \mathbb{N}_n, \sharp \Lambda=s_k\}$. We will first prove that for any $\Lambda\in\mathcal{O}$, $\sharp (\Omega'_\Lambda\cap \Omega_k)\le 1$,
where  $\Omega'_\Lambda:=\{x': x'_{\mathbb{N}_n\backslash \Lambda}=\mathbf{0}_{n-s_k}, x'\in\mathbb{R}^n\}$.

For  $\Lambda\in\mathcal{O}$, if there doest not exist $x\in\Omega_k$ such that $\mathrm{supp}(x)=\Lambda$, then we have $\Omega'_\Lambda\cap\Omega_k=\emptyset$, that is $\sharp(\Omega'_\Lambda\cap\Omega_k)=0$. Otherwise, if there exists  $x\in\Omega_k$ such that $\mathrm{supp}(x)=\Lambda$, we then prove $\Omega'_\Lambda\cap\Omega_k=\{x\}$.
 By the definition of $\Omega_k$ we have $\phi(\|Ax'-b\|_2)=\rho_k$ for any $x'\in\Omega_k$. By Item (ii) of Lemma \ref{lemma:Lambdastrict}, we have that $\|Ay-b\|_2>t$ for all $x\neq y\in \Omega'_\Lambda$. Then, we get $\phi(\|Ay-b\|_2)\neq \rho_k$ for all $x\neq y\in \Omega'_\Lambda$ due to $\sharp \{z: \phi(z)=\rho_k\}=1$. Thus, $\Omega'_\Lambda\cap \Omega_k=\{x\}$, that is $\sharp (\Omega'_\Lambda\cap \Omega_k)=1$.  Therefore, we get $\sharp (\Omega'_\Lambda\cap \Omega_k)\le 1$ for any $\Lambda\in\mathcal{O}$.

 Further, since $\Omega_k\subseteq \bigcup _{\Lambda\in\mathcal{O}}\Omega'_\Lambda$ and $\sharp \mathcal{O}=C_{n}^{s_k}$, we get \eqref{eq:15} of this proposition.
\end{proof}

The next corollary is obtained immediately.
\begin{corollary}\label{Coro:cardiP2}
For $\phi$ satisfying H\ref{hypo:1} and $p=2$,  let $L$ and $s_i, \rho_i, \Omega_i$ for $i\in\mathbb{N}_{L}^0$ be defined by Definition \ref{def:sRhoOmega}. If $\phi$ is strictly increasing, then
$\sharp \Omega_k\le C_n^{s_k}
$
for any $k\in \mathbb{N}_{L}^0$.
\end{corollary}
\subsection{Cardinality of $\Omega_k$ when $p=1$}
We first present an important lemma. To this end, we define, for any $x\in\mathbb{R}^d$ and $\delta>0$,  $U(x; \delta):=\{x': \|x-x'\|_2\le \delta, x'\in\mathbb{R}^d\}$.
For convenience, we adopt an assumption for a vector in $\mathbb{R}^m$.
\begin{hypothesis}\label{hypo:2}
For a vector $a\in\mathbb{R}^m$, there holds $\sum_{i=1}^m z_i a_i\neq 0$  for all $z\in\{z: z_i=\pm 1, i\in\mathbb{N}_m\}$.
\end{hypothesis}

\begin{lemma}\label{lema:l1}
Let $B$ an $m\times d$ real matrix and $b\in\mathbb{R}^m$. Let $\hat y\in\mathbb{R}^d$ satisfy   $\|\hat y\|_0=d$. Then
\begin{itemize}
\item [(i)] if $d=1$ and $B$ satisfies H\ref{hypo:2}, then $\sharp \hat S\le 2^m$, where $\hat S:=\{y: \|By-b\|_1=\|B\hat y-b\|_1, y\in\mathbb{R}\}$;
\item [(ii)] if $d\ge 2$ and $\|B\hat y-b\|_0\ge m+2-d$, then
for any $\delta>0$, there exists $\hat y\neq y\in U(\hat y; \delta)$ such that $\|By-b\|_1=\|B\hat y-b\|_1$.
\end{itemize}
\end{lemma}
\begin{proof}
For simplicity, we set $\hat \sigma:=\|B\hat y-b\|_1$.
We first prove Item (i). Let $O:=\{z: z_i=\pm 1, i\in\mathbb{N}_m\}$. Obviously $\sharp O=2^m$. Then for any $z\in O$, we set $S_z:=\{y: \sum_{i=1}^m z_i(By-b)_i=\hat\sigma, y\in\mathbb{R}\}$. It is clear that $\sharp S_z= 1$ due to H\ref{hypo:2}. Then we have $\hat S\subseteq \bigcup_{z\in  O} S_z$. Therefore, $\sharp \hat S\le 2^m$.  We get Item (i).

We next prove Item (ii).
Since $\|B\hat y-b\|_0\ge m+2-d$, we set $z:=\mathrm{sign}(B\hat y-b)$, where $\mathrm{sign}(\cdot)$ is the component-wise signum function ($0$ is assigned to $0$). Set $\Lambda:=\mathrm{supp}(z)$. We have $\sharp \Lambda\ge m+2-d$. Then  there exists $\delta_1>0$ such that for any $y\in U(\hat y; \delta_1)$, there holds  $\mathrm{sign}(By-b)_i=z_i$ for any $i\in\Lambda$.
Let $S:=\{y: \sum_{i\in\Lambda} z_i(By-b)_i=\hat \sigma, (By-b)_j=0 \mathrm{~for~} j\in\mathbb{N}_m\backslash \Lambda, y\in\mathbb{R}^d\}$ and $S_0:=\{y: \sum_{i\in\Lambda} z_i(By)_i=0,  (By)_j=0 \mathrm{~for~} j\in\mathbb{N}_m\backslash \Lambda, y\in\mathbb{R}^d\}$. Clearly, $\hat y\in S$. Therefore, we have $S=\{y: y=\hat y+\tilde y, \tilde y\in S_0\}$. Obviously, $S_0$ is an Euclid linear space with dimension no less then $d-(m-\sharp \Lambda+1)\ge 1$. Thus, for any $\delta>0$, there exists $0\neq y^*\in S_0$ such that  $\| y^*\|_2\le \min\{\delta_1, \delta\}$. Set $y=\hat y+y^*$. Then, we have $y\in U(\hat y; \delta)$, $\mathrm{sign}(By-b)_i=z_i$ for $i\in\Lambda$ and $y\in S$. Thus, we have $\|By-b\|_1=\sum_{i\in\Lambda}z_i(By-b)_i=\hat\sigma$. Then, we get this lemma.
\end{proof}

Now, we are ready to present a proposition on the cardinality of $\Omega_k$ for $p=1$.
\begin{proposition}\label{prop:cardiP1}
For $\phi$ satisfying H\ref{hypo:1} and $p=1$,  let $L$ and $s_i, \rho_i, \Omega_i$ for $i\in\mathbb{N}_{L}^0$ be defined by Definition \ref{def:sRhoOmega}.  For $k\in \mathbb{N}_{L}^0$, the following statements hold:
\begin{itemize}
\item [(i)] For $k=L$, $\sharp \Omega_k=1$.
\item [(ii)] For $k=L-1$, if $s_{L-1}=1$, all the columns of $A$ satisfy H\ref{hypo:2} and $\sharp\{z: \phi(z)=\rho_{L-1}\}=1$, then $\sharp \Omega_{L-1}\le n2^m$.
\item [(iii)] For $k=L-1$ with $s_k\ge 2$ or $k\le L-2$, if there exists $x\in\Omega_k$ such that $\|Ax-b\|_0\ge m+2-s_k$, then for any $\delta>0$ there exists $x\neq x'\in U(x; \delta)$ such that $x'\in\Omega_k$, therefore,  $\sharp\Omega_k=+\infty$.
\end{itemize}
\end{proposition}
\begin{proof}
Item (i) follows from the fact that $\Omega_L=\{\mathbf{0}_{n}\}$.
We next prove Item (ii). By the definition of $\Omega_{L-1}$ and $s_{L-1}=1$, we have $\|x\|_0=1$ and $\phi(\|Ax-b\|_1)=\rho_{L-1}$ for any $x\in\Omega_{L-1}$. Since $\sharp \{z:\phi(z)=\rho_{L-1}\}=1$, then we assume $z^*$ satisfies $\phi(z^*)=\rho_{L-1}$.   Let $O_i:=\{x: \|Ax-b\|_1=z^*, \mathrm{supp}(x)=\{i\}, x\in\mathbb{R}^n\}$ for $i\in\mathbb{N}_n$. Then $\Omega_{L-1}\subseteq\bigcup _{i\in\mathbb{N}_n} O_i$. By Item (i) of Lemma \ref{lema:l1}, we have $\sharp O_i\le 2^m$ for $i\in\mathbb{N}_n$. Then we get Item (ii) of this proposition.

We finally prove Item (iii). Let $\Lambda:=\mathrm{supp}(x)$ and $\hat\sigma:=\|Ax-b\|_1=\|A_\Lambda x_\Lambda-b\|_1$. Since $x\in\Omega_k$, we have $\|x\|_0=s_k$ and $\phi(\hat\sigma)=\rho_k$. Set $\delta_1:=\frac{1}{2}\min\{|x_i|: i\in\Lambda\}$. For any $\delta>0$, set $\delta^*:=\min\{\delta_1, \delta\}$. Then by Item (ii) of Lemma \ref{lema:l1},   there exists $x\neq x'\in U(x; \delta^*)$ such that $\mathrm{supp}(x')=\Lambda$ and $\|Ax'-b\|_1=\hat\sigma$.   Then we have $x'\in U(x; \delta)$ and $x'\in \Omega_k$ by the definition of $\Omega_k$.  Thus $\sharp \Omega_k=+\infty$.  Item (iii) follows.
\end{proof}

\subsection {Cardinality of optimal solution set of problem \eqref{prob:constrain} when $p=2$}
This subsection is devoted to the cardinality of the optimal solution set of problem \eqref{prob:constrain} when $p=2$.

We begin with recalling the notion of strict minimizers.
\begin{definition}
For a function $g:\mathbb{R}^d\rightarrow \mathbb{R}$ and a set $S\subseteq\mathbb{R}^d$, $\hat x\in\mathbb{R}^d$ is a strict minimizer of the problem $\min\{g(x): x\in S\}$ if there exists $\delta>0$ such that for any $\hat x\neq x\in U(\hat x; \delta)\cap S$ there holds $g(x)>g(\hat x)$.
\end{definition}

It is clear that as $\sigma\ge \|b\|_2$, $\sharp\widehat\Omega(\sigma)=1$ and $\mathbf{0}_n$ is the unique minimizer of problem \eqref{prob:constrain}. We next only consider the case when $\sigma<\|b\|_2$.

\begin{theorem}\label{thm:strictP2}
For $\phi:\mathbb{R}_+\rightarrow\mathbb{R}$ defined at $z\ge 0$ as $\phi(z)=z$ and $p=2$,  let $L$ and $s_i, \rho_i, \Omega_i$ for $i\in\mathbb{N}_{L}^0$ be defined by Definition \ref{def:sRhoOmega}. Let $\sigma \in [\sigma^*, \|b\|_2)$, where $\sigma^*$ is defined by \eqref{def:sigmaStar}. Let $\widehat\Omega$ be defined by \eqref{def:omegaSharp}. Then the following statements are equivalent:
\begin{itemize}
\item [(i)] There exits $k\in \mathbb{N}_{L-1}^0$ such that $\rho_k=\sigma$.
\item [(ii)] $\|Ax-b\|_2=\sigma$ for any $x\in\widehat \Omega(\sigma)$.
\item [(iii)]  Any $x\in\widehat\Omega(\sigma)$ is a strict minimizer of problem \eqref{prob:constrain}.
\item [(iv)] $\sharp \widehat\Omega(\sigma)$ is finite.
\end{itemize}
\end{theorem}

\begin{proof}
We first prove Item (i) implies Item (iv). By Item (ii) of Theorem \ref{thm:staConst}, we have $\widehat\Omega(\sigma)=\Omega_k$ since $\rho_k=\sigma$. Then we have $\sharp \widehat\Omega(\sigma)$ is finite due to Corollary \ref{Coro:cardiP2}. Thus, we get Item (iv).

It is obvious that Item (iv) implies Item (iii).

We next prove Item (iii) implies Item (ii) by contradiction.
Suppose there exists $x^*\in\widehat\Omega(\sigma)$ such that $\|Ax^*-b\|_2<\sigma$. We will show $x^*$ is not a strict minimizer of problem \eqref{prob:constrain}.  Let $\Lambda:=\mathrm{supp}(x^*)$. Then there exists $\delta_1>0$ such that for any $x\in U(x^*; \delta_1)$ there holds $\|Ax-b\|_2< \sigma$. Let $\delta_2: =\frac{1}{2}\{|x_i^*|: i\in\Lambda\}$. Set $\delta^*:=\min\{\delta_1, \delta_2\}$. Then for any $x\in U(x^*; \delta^*)$ we have $x_i\neq 0$ for $i\in\Lambda$ and $\|Ax-b\|_2<\sigma$.  Thus, $x\in \widehat\Omega(\sigma)$  for any $x\in U(x^*, \delta^*)\cap \{x: \mathrm{supp}(x)=\Lambda\}$. Clearly, $\sharp(U(x^*; \delta^*)\cap \{x: \mathrm{supp}(x)=\Lambda\})=+\infty$. Then, for any $\delta>0$, there exists $x\in U(x^*, \delta)\cap G_\sigma$, such that $\|x\|_0=\|x^*\|_0$, contradicting the fact that $x^*$ is a strict minimizer of problem \eqref{prob:constrain}. Therefore, we obtain that Item (iii) implies Item (ii).

We finally show that Item (ii) implies Item (i). We also prove it by contradiction. Suppose Item (i) doest not hold. Then by Item (i) of Proposition \ref{prop:partialEauvalence}, there exists $\rho_j<\sigma<\rho_{j+1}$, where $j\in \mathbb{N}_{L-1}^0$ and $\Omega_j\subset\widehat\Omega(\sigma)$ due to Item (iii) of  Proposition \ref{prop:partialEauvalence}. By Item (ii), $\|Ax-b\|_2=\sigma>\rho_j$ for any $x\in\Omega_j$, contradicting the fact that $\|Ax-b\|_2=\rho_j$ due to the definition of $\Omega_j$. Thus we get Item (ii) implies Item (i). Then we complete the proof.
\end{proof}

The next corollary follows from Theorem \ref{thm:strictP2}.
\begin{corollary}
For  $\phi:\mathbb{R}_+\rightarrow\mathbb{R}$ defined at $z\ge 0$ as $\phi(z)=z$ and $p=2$,  let $L$ and $s_i, \rho_i, \Omega_i$ for $i\in\mathbb{N}_{L}^0$ be defined by Definition \ref{def:sRhoOmega}. For $\sigma\in(\rho_k, \rho_{k+1})$, $k\in\mathbb{N}_{L-1}^0$, the following hold:
\begin{itemize}
\item [(i)] There exits $x\in\widehat \Omega(\sigma)$ such that $x$ is not a strict minimizer of problem \eqref{prob:constrain}.
\item [(ii)] $\sharp\widehat \Omega(\sigma)=+\infty$.
\end{itemize}
\end{corollary}

\subsection{Cardinality of optimal solution set of problem \eqref{prob:penalty} when $p=2$}
In this subsection, we discuss cardinality of optimal solution set of problem \eqref{prob:penalty} when $p=2$.

We state the main results in the following theorem.
\begin{theorem}\label{thm:penaCarp2}
For $\phi$ satisfying H\ref{hypo:1} and $p=2$,  let $L$ and $s_i, \rho_i, \Omega_i$ for $i\in\mathbb{N}_{L}^0$ be defined by Definition \ref{def:sRhoOmega}. Let  $K$, $\{\Lambda_i\subseteq \mathbb{N}_{L}^0: i=-1,0,\dots, K-1\}$ and $t_i, \lambda_i$ for $i\in\mathbb{N}_{K}^0$ be  defined by  Definition \ref{def:lambda_i}.
Let $\lambda_{-1}:=+\infty$, $S_i:=\{z:\phi(z)=\rho_{i}\}$ for $i \in\mathbb{N}_{L}^0$ and $\Omega$ defined by \eqref{def:Omega}.
Then  for $k\in\mathbb{N}_K^0$, the following statements hold:
\begin{itemize}
\item [(i)] If $\sharp S_{t_k}=1$, then   for any $\lambda\in (\lambda_k, \lambda_{k-1})$, $\sharp \Omega(\lambda)$ is finite and any $x\in\Omega(\lambda)$ is a strict minimizer of problem \eqref{prob:penalty}.
\item [(ii)] If $k\neq K$ and $\sharp S_j=1$ for all $j\in\Lambda_k\cup\{t_k\}$, then  $\sharp \Omega(\lambda_k)$ is finite and any $x\in\Omega(\lambda_k)$ is a strict minimizer of problem \eqref{prob:penalty}.
\end{itemize}
\end{theorem}
\begin{proof}
We first prove Item (i). Since $\sharp S_{t_k}=1$,  $\sharp \Omega_{t_k}$ is finite. By (ii) of Theorem \ref{thm:stability}, we have $\Omega(\lambda)=\Omega_{t_k}$ for any $\lambda\in(\lambda_k, \lambda_{k-1})$. Then we get Item (i).

We then prove Item (ii). Since $\sharp S_j=1$ for all $j\in\Lambda_k\cup \{t_k\}$, $\sharp \Omega_{j}$ is finite for all $j\in\Lambda_k\cup \{t_k\}$. Then, by Item (ii) of Theorem \ref{thm:stability}, $\sharp \Omega(\lambda_k)$ is finite. Then we get Item (ii) of this theorem.
\end{proof}

A direct and useful consequence of Theorem \ref{thm:penaCarp2} is presented below.
\begin{corollary}\label{coro:penaCarp2}
 Let  $\phi$ satisfy H1 and $p=2$. If $\phi$ is strictly increasing, then $\sharp\Omega(\lambda)$ is finite and every $x\in\Omega(\lambda)$ is a strict minimizer of problem \eqref{prob:penalty} for any $\lambda>0$.
\end{corollary}

With the help of the above corollary, we characterize the cardinality of optimal solution set of problem \eqref{prob:l2-unconstrain} in the next proposition. Our result coincides with that obtained in \cite{Nikolova:2013}.
\begin{proposition}
The cardinality of the optimal solution set of problem \eqref{prob:l2-unconstrain} is finite for any $\lambda>0$ and every optimal solution is a strict minimizer.
\end{proposition}
\begin{proof}
Problem \eqref{prob:l2-unconstrain} is a special case of problem \eqref{prob:penalty} with $\phi(z):=\frac{z^2}{2}$ for  $z\ge 0$. Since in this case $\phi$ is strictly increasing, we obtain the conclusions by Corollary \ref{coro:penaCarp2}.
\end{proof}

\subsection{Cardinality of optimal solution sets of problems \eqref{prob:constrain} and \eqref{prob:penalty} when $p=1$}
This subsection provides results on cardinality of optimal solution sets of problems \eqref{prob:constrain} and \eqref{prob:penalty} when $p=1$.

\begin{theorem}\label{thm:ProbCarp1}
For $\phi$ satisfying H\ref{hypo:1} and $p=1$, let $L$ and $s_i, \rho_i, \Omega_i$ for $i\in\mathbb{N}_{L}^0$ be defined by Definition \ref{def:sRhoOmega}. Let  $K$, $\{\Lambda_i\subseteq \mathbb{N}_{L}^0: i=-1,0,\dots, K-1\}$ and $t_i, \lambda_i$ for $i\in\mathbb{N}_{K}^0$ be  defined by  Definition \ref{def:lambda_i}. Let $\lambda_{-1}:=+\infty$. Let $\Omega$ and $\widehat\Omega$ be defined by \eqref{def:Omega} and \eqref{def:omegaSharp} respectively. Then  the following statements hold:
\begin{itemize}
\item [(i)] $\widehat \Omega(\sigma)=\Omega(\lambda)=\{\mathbf{0}_n\}$ for $\sigma\ge \|b\|_1$ and $0<\lambda<\lambda_{K-1}$.
\item [(ii)] For $k=K-1$ with $s_{t_k}\ge2$ or $k\le K-2$, if there exists $x^*\in \Omega_{t_k}$ such that $\|Ax^*-b\|_0\ge m+2-s_{t_k}$, then we have $\sharp \Omega(\lambda)=+\infty$  for any $\lambda\in[\lambda_k, \lambda_{k-1})$ and $x^*$ is not a strict minimizer of problem \eqref{prob:penalty}.
\item [(iii)]  Suppose $\phi(z)=z$ for $z\ge 0$, then $\sharp \widehat\Omega(\sigma)=+\infty$ for all $\sigma\in(\rho_k, \rho_{k+1})$, $k\in\mathbb{N}_{L-1}^0$; further, for $\sigma=\rho_k$ and $k$ satisfying $k=K-1$ with $s_{k}\ge2$ or $k\le K-2$, if there exists $x^*\in \Omega_{k}$ such that $\|Ax^*-b\|_0\ge m+2-s_{k}$, then we have  $\sharp\widehat\Omega(\sigma)=+\infty$, and $x^*$ is not a strict minimizer of problem \eqref{prob:constrain}.
\end{itemize}
\end{theorem}
\begin{proof}
By Item (ii) of Theorem \ref{thm:stability} and the definition of $\Omega_L$, Item (i) follows immediately. By Theorem \ref{thm:stability}, $\Omega_{t_k}\subseteq\Omega(\lambda)$ for $\lambda\in[\lambda_k, \lambda_{k-1})$. Then we have Item (ii) due to Item (iii) of Proposition \ref{prop:cardiP1}.
For $\sigma\in(\rho_k, \rho_{k+1})$, $k\in\mathbb{N}_{L-1}^0$, from Item (ii) of Theorem \ref{thm:staConst}, $\Omega_{k}\subseteq\widehat\Omega(\sigma)$.
Then any $x\in\Omega_{k}\subseteq \widehat\Omega(\sigma)$ satisfies $\|Ax-b\|_1=\rho_k<\sigma$. Let $x^*\in\Omega_k$. Similar to the proofs of Theorem \ref{thm:strictP2}, for any $\delta>0$, there exits $y\neq x^*$ such that $\|Ay-b\|_1\le \sigma$ and $\supp (y)=\supp (x^*)$. Thus, $y\in \widehat\Omega(\sigma)$. Therefore, $\sharp\widehat\Omega(\sigma)=+\infty$ for $\sigma\in(\rho_k, \rho_{k+1})$. For $\sigma=\rho_k$, we have $\Omega_k=\widehat\Omega(\sigma)$. Then by Item (iii) of Proposition \ref{prop:cardiP1} we obtain Item (iii) of this theorem.
\end{proof}
\section{Numerical illustrations}\label{sec:exp}
In this section, we provide numerical illustrations for some of our main theoretical findings, including:
\begin{itemize}
\item the marginal function of problem \eqref{prob:l2-constrain} is piecewise constant, while that of problem \eqref{prob:l2-unconstrain} is piecewise linear;
\item it is possible that the least square penalty problem \eqref{prob:l2-unconstrain} has no common optimal solutions with the constrained problem  \eqref{prob:l2-constrain} for any $\lambda>0$;
\item there exists $\lambda^*>0$ such that problem \eqref{prob:l2-env}  is an exact penalty formulation of the constrained problem \eqref{prob:l2-constrain} for $\lambda>\lambda^*$.
\end{itemize}
All the $\ell_0$ optimization problems involved  are solved by an exhaustive combinational search.
For better readability, all numerical digits are round to four decimal places.

We present in detail a concrete example for sparse signal recovery with $(m,n)=(4,5)$ as follows
\begin{eqnarray}
\label{eq:A}
A=\left [
\begin{array}{ccccc}
1&7&7&7&9\\
2&1&3&7&2\\
3&3&3&9&3\\
3&4&9&8&9
\end{array}
\right],\\
\label{eq:x}
x^*=\left [
\begin{array}{ccccc}
0 & 1 & 1 &0 & 0\\
\end{array}
\right]^\top,\\
\label{eq:b}
b=Ax+\eta=
\left[\begin{array}{ccccc}
14.43 & 7.21 & 4.49 & 13.02
\end{array}\right]^\top,
\end{eqnarray}
where $\eta\in\mathbb{R}^4$ denotes random noise following a nearly normal distribution. The coefficients of $A$, $x^*$, and $b$ in \eqref{eq:A},\eqref{eq:x}, and \eqref{eq:b} are exact. The noise level of $\eta$ is $\|Ax^*-b\|_2=3.5734$. As $x^*$ is sparse, we naturally expect recovering $x^*$ by solving problem \eqref{prob:l2-constrain} with $A$, $b$ in \eqref{eq:A}, \eqref{eq:b} and $\sigma=3.6$. To this end, we shall first verify that $x^*$ is an optimal solution to problem \eqref{prob:l2-constrain}. Let $\phi(z)=z$, $z\geq 0$ and $p=2$. Then according to Definition \ref{def:sRhoOmega}, we have that $L=\mathrm{rank}(A)=4$,
\begin{equation}\label{eq:s}
s_0=4, s_1=3, s_2=2, s_3=1, s_4=0,
\end{equation}
and
\begin{equation}\label{eq:rho}
\rho_0=0, \rho_1=1.4487, \rho_2=3.3363, \rho_3=4.0502,\rho_4=21.2106.
\end{equation}
By Theorem \ref{thm:staConst}, the marginal function of problem \eqref{prob:l2-constrain}  has the form of
$$
H(\sigma)=\begin{cases}
4,&\mathrm{~~if~~}0\le \sigma<1.4487,\\
3,&\mathrm{~~if~~}1.4487\le \sigma<3.3363,\\
2,&\mathrm{~~if~~}3.3363\le \sigma<4.0502,\\
1,&\mathrm{~~if~~}4.0502\le \sigma<21.2106,\\
0,&\mathrm{~~if~~}\sigma\ge 21.2106.
\end{cases}
$$
Specially, we find that $H(\sigma)=2$ when $ 3.3363\le\sigma < 4.0502$. This together with $\|x^*\|=2$ implies that $x^*$ is an optimal solution to problem \eqref{prob:l2-constrain} with $\sigma=3.6$.

Next we consider the least square penalty problem \eqref{prob:l2-unconstrain} with $A$, $b$ in \eqref{eq:A} and \eqref{eq:b}. By Remark \ref{remark:same}, in this case the marginal function of problem \eqref{prob:l2-unconstrain} is
\begin{equation}\label{eq:marginf}
F(\lambda)=\min\{s_i+\frac{\rho_i^2}{2}\lambda:i\in\mathbb{N}^0_L\},
\end{equation}
where $L=4$, $s_i$, $\rho_i$, $i\in\mathbb{N}^0_L$ are given in \eqref{eq:s} and \eqref{eq:rho}.
Then, $K$, $t_i, \lambda_i$ for $i\in\mathbb{N}_K^0$, and  $\Lambda_i$ for $i\in\{-1,0,\dots, K-1\}$ are produced by the iterative procedure in Definition \ref{def:lambda_i}. We obtain that $K=3$,
$$
t_0=0, t_1=1, t_2=3, t_3=4,
$$
$$
\lambda_0=0.9530, \lambda_1=0.2796, \lambda_2=0.0046, \lambda_3=0,
$$
$$
\Lambda_{-1}=\{0\}, \Lambda_0=\{1\}, \Lambda_1=\{3\}, \Lambda_2=\{4\}.
$$
The plots of $f_i(\lambda):=s_i+\frac{\rho_i^2}{2}\lambda$ for $i\in\mathbb{N}_4^0$ and their intersection points are depicted in Figure \ref{fig:ex2}.
\begin{figure}
\centering
\scalebox{0.8}{\includegraphics{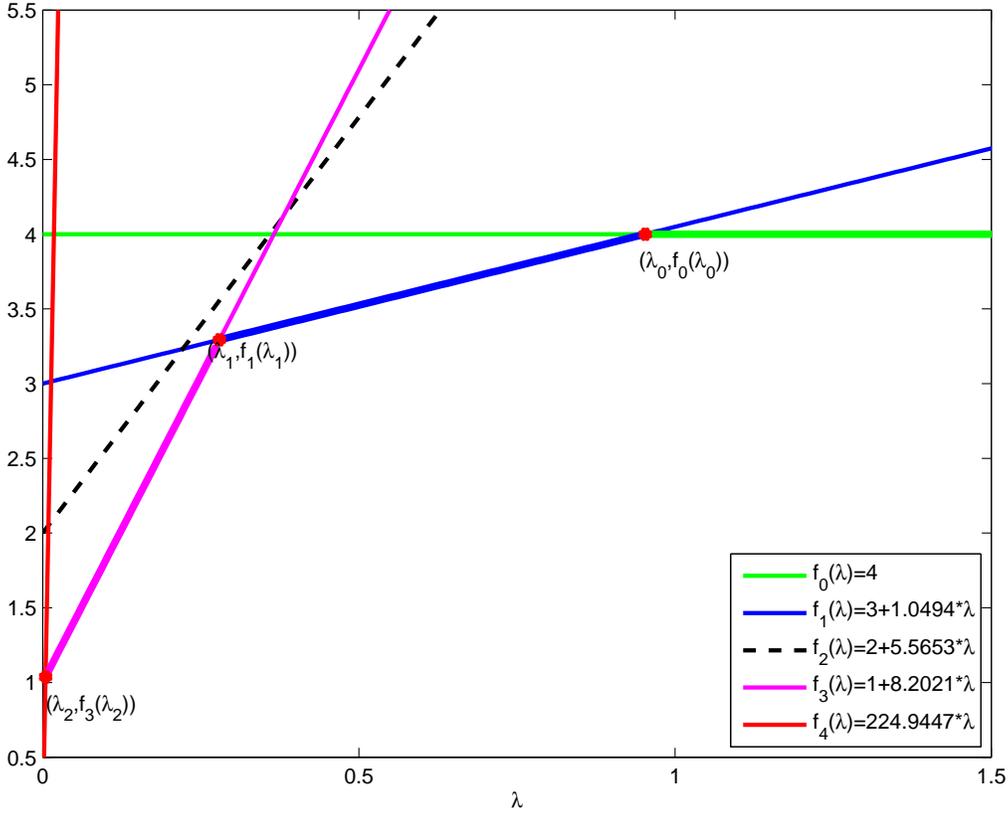}}
\caption{Plots of $f_i(\lambda):=s_i+\frac{\rho_i^2}{2}\lambda$ for $i\in\mathbb{N}_4^0$ in \eqref{eq:marginf}. The line in bold represents $F(\lambda)$ in \eqref{eq:marginf}. }
\label{fig:ex2}
\end{figure}
By Theorem \ref{thm:stability}, $F(\lambda)$ in \eqref{eq:marginf} reads as
$$
F(\lambda)=\begin{cases}
224.9447\lambda,&\mathrm{~~if~~}0< \lambda<0.0046,\\
1+8.2021\lambda,&\mathrm{~~if~~}0.0046\le \lambda<0.2796,\\
3+1.0494\lambda,&\mathrm{~~if~~}0.2796\le \lambda<0.9530,\\
4,&\mathrm{~~if~~}\lambda\ge 0.9530.
\end{cases}
$$
We note that $F(\lambda)\neq  s_2+\frac{\rho^2_2}{2}\lambda$ for any $\lambda>0$. By Theorem \ref{thm:stability}, we deduce that $\|x\|_0\neq 2$ for any $x\in\Omega(\lambda)$ and any $\lambda>0$. Therefore, given $A$, $b$ in \eqref{eq:A}, \eqref{eq:b}, problem \eqref{prob:l2-unconstrain} has no common optimal solution as problem \eqref{prob:l2-constrain} with $\sigma=3.6$ for any $\lambda>0$. In other words, problem \eqref{prob:l2-unconstrain} fails to recover or approximate the true signal $x^*$.

In the next numerical test, we discuss penalty problem \eqref{prob:l2-env} with $A$, $b$ in \eqref{eq:A}, \eqref{eq:b} and $\sigma=3.6$. As shown in Section \ref{sec:equal}, problem \eqref{prob:l2-env} has the same optimal solution set with problem \eqref{prob:l2-constrain}, provided that $\lambda>\lambda^*$ for some $\lambda^*>0$. We shall find the exact numerical value of $\lambda^*$. Set $\phi=\phi_2$ as in \eqref{def:phi2}. From the iteration in Definition \ref{def:sRhoOmega}, we have that $\tilde{L}=2$,
\begin{equation}\label{eq:s-2}
\tilde s_0=2, \tilde s_1=1, \tilde s_2=0,
\end{equation}
\begin{equation}\label{eq:rho-2}
\tilde \rho_0=0, \tilde \rho_1=0.1013, \tilde \rho_2=155.0666.
\end{equation}
According to the proof of Theorem \ref{thm:equivalence}, $\lambda^*$ actually equals to $\tilde \lambda_0$ produced by the iteration procedure in Definition \ref{def:lambda_i}. Given $\tilde s_i$, $\tilde \rho_i$ in \eqref{eq:s-2}, \eqref{eq:rho-2}, $i\in\mathbb{N}^0_2$, it holds that
$$
\tilde \lambda_0 = \max\{\frac{\tilde s_0-\tilde s_1}{\tilde \rho_1-\tilde \rho_0},\frac{\tilde s_0 -\tilde s_2}{\tilde \rho_2-\tilde \rho_0}\}=9.8678.
$$
Therefore, we conclude that for $A$, $b$ in \eqref{eq:A}, \eqref{eq:b} and $\sigma=3.6$, penalty problem \eqref{prob:l2-env} share the same optimal solution set as constrained problem \eqref{prob:l2-constrain} when the penalty parameter $\lambda>\lambda^*=9.8678$.

\section{Conclusions}\label{sec:con}
The quadratically constrained $\ell_0$ regularization problem \eqref{prob:l2-constrain} arises in many applications where sparse solutions are desired. The least square penalty method \eqref{prob:l2-unconstrain} has been widely used for solving this problem. There is little results regarding properties of optimal solutions to problems \eqref{prob:l2-constrain} and \eqref{prob:l2-unconstrain}, although these two problems have been well studied mainly from a numerical perspective. In this paper, we aim at investigating optimal solutions of a more general constrained $\ell_0$ regularization problem \eqref{prob:constrain} and the corresponding penalty problem \eqref{prob:penalty}.  We first discuss the optimal solutions of these two problems in detail, including existence, stability with respect to parameter, cardinality and strictness. In particular, we show that the optimal solution set of the penalty problem is piecewise constant with respect to the penalty parameter. Then we clarify the relationship between optimal solutions of the two problems. It is proven that, when the noise level $\sigma>0$ there does not always exist a penalty parameter, such that  problem \eqref{prob:l2-unconstrain} and the constrained $\ell_0$ problem
\eqref{prob:l2-constrain} share a common optimal solution. Under mild conditions on the penalty function, we prove that  problem \eqref{prob:penalty} has the same optimal solution set with
 problem \eqref{prob:constrain} if the penalty parameter is large enough. With the help of the conditions, we further propose exact penalty problems for the constrained problem. We expect our theoretical findings can offer motivations for the design of innovative numerical schemes for solving constrained $\ell_0$ regularization problems.

\section{Appendix: Proof of stability for problem \eqref{prob:penalty} in section \ref{sec:staP}}

\subsection{Proof of Proposition \ref{prop:lambda_i}}
We begin with an auxiliary Lemma.
\begin{lemma}\label{lemma:Lam}
Given $\phi$ satisfying H\ref{hypo:1} and $p\in\{1, 2\}$, let $L$ and $s_i, \rho_i, \Omega_i$ for $i\in\mathbb{N}_{L}^0$ be defined by Definition \ref{def:sRhoOmega}. Let  $K$, $\{\Lambda_i\subseteq \mathbb{N}_{L}^0: i=-1,0,\dots, K-1\}$ and $t_i, \lambda_i$ for $i\in\mathbb{N}_{K}^0$ be  defined by  Definition \ref{def:lambda_i}. Then
 $$\Lambda_i=\arg\min\{s_j+\lambda_i\rho_j: j=t_{i}+1,\dots, L\}.$$
\end{lemma}
\begin{proof}
By the definition of $\lambda_i$ and $\Lambda_i$, for any $k\in\Lambda_i$,  $\lambda_i=\frac{s_{t_i}-s_k}{\rho_k-\rho_{t_i}}\ge\frac{s_{t_i}-s_j}{\rho_j-\rho_{t_i}}$ for  $j=t_{i+1}+1,\dots, L$. It follows that for any $k\in\Lambda_i$, $s_k+\lambda_i\rho_k\le  s_j+\lambda_i\rho_j $ for  $j=t_{i}+1,\dots, L$. Therefore, we have $\Lambda_i\subseteq \arg\min\{s_j+\lambda_i\rho_j: j=t_{i}+1,\dots, L\}$.

Conversely, for any $k\in \arg\min\{s_j+\lambda_i\rho_j: j=t_{i}+1,\dots, L\}$,  $s_k+\lambda_i\rho_k\le s_j+\lambda_i\rho_j$ for $j=t_{i}+1,\dots, L$. By the definition of $\lambda_i$,  $s_{t_i}+\lambda_i\rho_{t_i}\le s_j+\lambda_i\rho_j$ for $j=t_{i}+1,\dots, L$ and there exists $k^*\in\{t_{i}+1,\dots, L\}$ such that  $s_{t_i}+\lambda_i\rho_{t_i}= s_{k^*}+\lambda_i\rho_{k^*}$. Thus, $s_k+\lambda_i\rho_k=s_{t_i}+\lambda_i\rho_{t_i}$ and $\lambda_i=\frac{s_{t_i}-s_k}{\rho_k-\rho_{t_i}}\ge\frac{s_{t_i}-s_j}{\rho_j-\rho_{t_i}} $ for $j=t_{i}+1,\dots, L$. By the definition of $\Lambda_i$, one has $k\in\Lambda_i$.
\end{proof}

Now we are ready to give the proof of Proposition \ref{prop:lambda_i}.
\begin{proof}
It suffices to prove Items (ii) and (iii)  since Item (i) can be obtained by Item (ii) and Item (iv) is clear from the definition of $\Lambda_i$.

We first prove Item (ii). By the definition, for any $i \in \mathbb{N}_{K-1}^0$, it holds that $t_i<y$ for any $y\in\Lambda_i$. Thus we get $t_i<t_{i+1}$ for $i \in \mathbb{N}_{K-1}^0$ by the fact that $t_{i+1}\in\Lambda_{i}$. Then we obtain Item (ii).

We then prove Item (iii). From the definition of $t_{i+1}$, $t_{i+1}\in\Lambda_i$ for $i \in \mathbb{N}_{K-1}^0$. Then by Lemma \ref{lemma:Lam}, it follows that $s_{t_{i+1}}+\lambda_i\rho_{t_{i+1}}<s_j+\lambda_i\rho_j$ for any $j=t_{i+1}+1,\dots, L$ due to $t_{i+1}=\max\Lambda_i$. Then we have $\lambda_i>\frac{s_{t_{i+1}}-s_j}{\rho_j-\rho_{t_{i+1}}}$ for $j=t_{i+1}+1,\dots, L$. According to the definition of $\lambda_{i+1}$, $\lambda_i>\lambda_{i+1}$ for $i \in \mathbb{N}_{K-1}^0$. Item (iii) follows immediately.
\end{proof}

\subsection{Proof of Theorem \ref{thm:stability}}

We first show a lemma  which  plays an important role in the following analysis.
\begin{lemma}\label{lema:lambdai}
Given $\phi$ satisfying H\ref{hypo:1} and $p\in\{1, 2\}$, let $L$ and $s_i, \rho_i, \Omega_i$ for $i\in\mathbb{N}_{L}^0$ be defined by Definition \ref{def:sRhoOmega}. Let  $K$, $\{\Lambda_i\subseteq \mathbb{N}_{L}^0: i=-1,0,\dots, K-1\}$ and $t_i, \lambda_i$ for $i\in\mathbb{N}_{K}^0$ be  defined by  Definition \ref{def:lambda_i}.
For $i \in \mathbb{N}_{K-1}^0$,
\begin{itemize}
\item [(i)] $s_{t_{i+1}}+\lambda_i\rho_{t_{i+1}}=s_{t_{i}}+\lambda_i\rho_{t_{i}}$,
\item [(ii)] $s_{t_{i+1}}+\lambda_{i}\rho_{t_{i+1}}\le s_{j}+\lambda_{i}\rho_{j}$ for $j=t_{i},t_{i}+1\dots, L$,
\item [(iii)] for $j\in\mathbb{N}_i^0$,
        \begin{equation}\label{eq:8.1}
        s_{t_{i+1}}+\lambda_{i+1}\rho_{t_{i+1}}<s_{t_{j}}+\lambda_{i+1}\rho_{t_{j}},
        \end{equation}
\item [(iv)] $\lambda_i\le \frac{s_j-s_{t_{i+1}}}{\rho_{t_{i+1}}-\rho_j}$ for all $j\in\mathbb{N}_{t_{i+1}-1}^0$, in particular,  $\lambda_i< \frac{s_j-s_{t_{i+1}}}{\rho_{t_{i+1}}-\rho_j}$ for all $j\in\mathbb{N}_{t_{i}-1}^0$.
\end{itemize}
\end{lemma}
\begin{proof}
We first prove Items (i) and (ii). From the definition of $\lambda_i$ and $t_{i+1}$, it follows that $\lambda_i=\frac{s_{t_i}-s_{t_{i+1}}}{\rho_{t_{i+1}}-\rho_{t_i}}$, which implies Item (i). Item (ii) follows from Item (i), Lemma \ref{lemma:Lam} and $t_{i+1}\in\Lambda_i$.

Next we prove Item (iii).  Since $\lambda_{i+1}<\lambda_i$, $t_{i+1}>t_i$ and $\rho_{t_{i+1}}>\rho_{t_i}$, it holds that
\begin{eqnarray*}
s_{t_{i+1}}+\lambda_{i+1}\rho_{t_{i+1}}&=&s_{t_{i+1}}+\lambda_{i}\rho_{t_{i+1}}+(\lambda_{i+1}-\lambda_i)\rho_{t_{i+1}}\\
&<&s_{t_{i}}+\lambda_{i}\rho_{t_{i}}+(\lambda_{i+1}-\lambda_i)\rho_{t_{i}}\\
&=&s_{t_i}+\lambda_{i+1}\rho_{t_{i}}.\end{eqnarray*}
This means that \eqref{eq:8.1} holds for $j=i$. We assume for  $k\in\mathbb{N}_{i-1}^0$, \eqref{eq:8.1} holds when $j=k+1$. Then we try to show \eqref{eq:8.1} holds as $j=k$, that is \begin{equation}\label{eq:6}
s_{t_{i+1}}+\lambda_{i+1}\rho_{t_{i+1}}<s_{t_{k}}+\lambda_{i+1}\rho_{t_k}.
\end{equation}
One finds that
\begin{eqnarray*}
s_{t_{k+1}}+\lambda_{i+1}\rho_{t_{k+1}}&=&s_{t_{k+1}}+\lambda_{k}\rho_{t_{k+1}}+(\lambda_{i+1}-\lambda_k)\rho_{t_{k+1}}\\
&<&s_{t_{k}}+\lambda_{k}\rho_{t_{k}}+(\lambda_{i+1}-\lambda_k)\rho_{t_{k}}\\
&=&s_{t_k}+\lambda_{i+1}\rho_{t_{k}}.
\end{eqnarray*}
Then  \eqref{eq:6} follows. Therefore, we obtain Item (iii) immediately.

{{
Next, we prove Item (iv). Item (ii) implies that $\lambda_i\le \frac{s_j-s_{t_{i+1}}}{\rho_{t_{i+1}}-\rho_j}$ for all $j={t_i}, \dots, t_{i+1}-1$}}. Thus, we only need to prove $\lambda_i< \frac{s_j-s_{t_{i+1}}}{\rho_{t_{i+1}}-\rho_j}$ for  $j\in\mathbb{N}_{{t_i}-1}^0$, which is equivalent to
\begin{equation}\label{eq:3}
s_{t_{i+1}}+\lambda_i\rho_{{t_{i+1}}}< s_j+\lambda_i\rho_j
\end{equation}
for $j\in\mathbb{N}_{{t_i}-1}^0$.
By Item (ii), for $j\in\mathbb{N}_i$, it holds that
\begin{equation}
s_{t_j}+\lambda_{j-1}\rho_{t_j}\le s_k+\lambda_{j-1}\rho_k
 \end{equation}
for $k=t_{j-1}, \dots, t_{j}-1$.
From the fact that $0<\lambda_i< \lambda_{j-1}$ and $\rho_{t_j}>\rho_{k}>0$, we obtain
 \begin{equation}\label{eq:4}
s_{t_j}+\lambda_i\rho_{t_j}< s_k+\lambda_i\rho_k
 \end{equation}
 holds for $k=t_{j-1}, \dots, t_{j}-1$.
By Items (i) and (iii),  $s_{t_{i+1}}+\lambda_i\rho_{{t_{i+1}}}\le  s_{t_j}+\lambda_{i}\rho_{t_j}$ for $j\in\mathbb{N}_i$. This together with \eqref{eq:4} and Item (i) implies that
\eqref{eq:3} holds for $j\in\mathbb{N}_{t_{i}-1}^0$.
\end{proof}

The properties of $f_i$ defined in \eqref{def:fi} are described in the next proposition.

\begin{proposition}\label{prop:fi}
Given $\phi$ satisfying H\ref{hypo:1} and $p\in\{1, 2\}$, let $L$ and $s_i, \rho_i, \Omega_i$ for $i\in\mathbb{N}_{L}^0$ be defined by Definition \ref{def:sRhoOmega}. Let  $K$, $\{\Lambda_i\subseteq \mathbb{N}_{L}^0: i=-1,0,\dots, K-1\}$ and $t_i, \lambda_i$ for $i\in\mathbb{N}_{K}^0$ be  defined by  Definition \ref{def:lambda_i}. Let $f_i$ for $i \in\mathbb{N}_{L}^0$ \eqref{def:fi}.  Then the following statements hold:
\begin{itemize}
\item [(i)] For $k\in\mathbb{N}_{K-1}^0$, $f_i(\lambda_k)=f_{t_k}(\lambda_k)$ for all $i\in\Lambda_k$.
\item [(ii)] $f_{t_k}(\lambda_k)=f_{t_{k+1}}(\lambda_k)$ for $k\in\mathbb{N}_{K-1}^0$.
\item [(iii)]$f_0(\lambda)<f_j(\lambda)$ for any $\lambda\in(\lambda_0, +\infty)$ and any $j\in\mathbb{N}_L$.
\item [(iv)] $f_L(\lambda)<f_j(\lambda)$ for any $\lambda\in(0, \lambda_{K-1})$ and any $j\in\mathbb{N}_{L-1}^0$.
\item [(v)] For $k\in\mathbb{N}_{K-1}$, there holds $f_{t_k}(\lambda)<f_j(\lambda)$ for any $\lambda\in(\lambda_{k}, \lambda_{k-1})$ and any $j=1,2,\dots,t_k-1, t_{k}+1,\dots, L$.
\end{itemize}
\end{proposition}
\begin{proof}
We first prove Items (i) and (ii). From the definition of $\lambda_k$ and  $\Lambda_k$, for any $i\in\Lambda_k$ $\lambda_k=\frac{s_{{t_k}}-s_i}{\rho_i-\rho_{t_k}}$.  Item (i) is obtained immediately by the definition of $f_j$ for $j \in\mathbb{N}_{L}^0$ in \eqref{def:fi}. Item (ii)  follows from $t_{k+1}\in\Lambda_{k}$.

We next prove Item (iii). Thanks to Items (ii) and (iii) of Proposition \ref{prop:rhoSOmega}, Item (iii) amounts to
\begin{equation}\label{eq:7}
\lambda_0\ge \frac{s_0-s_j}{\rho_j-\rho_0}
 \end{equation}
 for all $j=1,2,\dots, L$. It is clear that \eqref{eq:7} holds by the definition of $\lambda_{0}$ in Definition \ref{def:lambda_i} and $t_0=0$.

 We then prove Item (iv). From Item (iv) of Lemma \ref{lema:lambdai} and $t_{K}=L$, $\lambda_{K-1}\le \frac{s_j-s_{L}}{\rho_L-\rho_j}$ for $j  \in\mathbb{N}_{L-1}^0$ and Item (iv) follows.

 Finally, we prove Item (v). By Item (iv) of Lemma \ref{lema:lambdai},  $\lambda <\frac{s_j-s_{t_k}}{\rho_{t_k}-\rho_j}$ for $\lambda<\lambda_{k-1}$ and $j\in\mathbb{N}_{t_k-1}^0$. Thus,  $f_{t_k}(\lambda)<f_j(\lambda)$ for any $\lambda< \lambda_{k-1}$ and any $j\in\mathbb{N}_{t_k-1}^0$. By the definition of $\lambda_k$, we have $\lambda>\frac{s_{t_k}-s_j}{\rho_j-\rho_{t_k}}$ for $\lambda>\lambda_k$ and $j=t_{k}+1,\dots,L$. Then, it follows that $f_{t_k}(\lambda)<f_j(\lambda)$ for any $\lambda> \lambda_{k}$ and any $j=t_{k}-1,\dots, L$.
\end{proof}

Now, we are ready to show Theorem \ref{thm:stability}.

\begin{proof}
We only need to prove $\Omega(\lambda_i)=\bigcup_{k\in\Lambda_i}\Omega_k\cup \Omega_{t_i}$ for $i\in\mathbb{N}_{K-1}^0$ in Item (ii), since  Item (i) and the rest of Item (ii) are direct consequences of  Theorem \ref{thm:exist} and Proposition \ref{prop:fi}.

By  Item (ii) of Theorem \ref{thm:exist} and Item (i) of Proposition \ref{prop:fi},  $\bigcup_{k\in\Lambda_{i}}\Omega_k\bigcup \Omega_{t_i}\subseteq\Omega(\lambda_i)$. Using this fact, it suffices to show $f_{t_i}(\lambda_i)<f_j(\lambda_i)$ for any $j\not\in\Lambda_{i}\cup\{t_i\}$. From Items (i) and (iv) of Lemma \ref{lema:lambdai}, $f_{t_i}(\lambda_i)=f_{t_i+1}(\lambda_i)<f_j(\lambda_i)$ for $j=0,1,\dots, t_i-1$. By the definition of $\Lambda_i$ and $\lambda_i$, we have $\lambda_i>\frac{s_{t_i}-s_j}{\rho_j-\rho_{t_i}}$ for $j\in\{t_i+1,\dots,L\}\backslash\Lambda_i$, which implies $f_{t_i}(\lambda_i)<f_j(\lambda_i)$ for $j\in\{t_i+1,\dots,L\}\backslash\Lambda_i$.
Hence $f_{t_i}(\lambda_i)<f_j(\lambda_i)$ for any $j\not\in\Lambda_{i}\cup\{t_i\}$.
\end{proof}
}

\vspace{10mm}

\bibliographystyle{siam}


\begin{thebibliography}{10}

\bibitem{attouch-Convergence}
{\sc Hedy Attouch, Jerome Bolte, and Benar~Fux Svaiter}, {\em Convergence of
  descent methods for semi-algebraic and tame problems: proximal algorithms,
  forward-backward splitting, and regularized gauss-seidel methods},
  Mathematical Programming, 137 (2013), pp.~91--129.

\bibitem{Teboulle:asympt2003}
{\sc Alfred Auslender and Marc Teboulle}, {\em Asymptotic Cones and Functions
  in Optimization and Variational Inequalities}, springer, 2003.

\bibitem{Bauschke-Combettes:11}
{\sc Heinz~H. Bauschke and Patrick~L. Combettes}, {\em Convex Analysis and
  Monotone Operator Theory in Hilbert Spaces}, AMS Books in Mathematics,
  Springer, New York, 2011.

\bibitem{BergFriedlander:SIAMO:2011}
{\sc Ewout van~den Berg and Michael~P. Friedlander}, {\em Sparse optimization
  with least-squares constraints}, SIAM Journal on Optimization, 21 (2011),
  pp.~1201--1229.

\bibitem{Bissantz2009Convergence}
{\sc Nicolai Bissantz, Lutz Mbgen, Axel Munk, and Bernd Stratmann}, {\em
  Convergence analysis of generalized iteratively reweighted least squares
  algorithms on convex function spaces}, SIAM Journal on Optimization, 19
  (2009), pp.~1828--1845.

\bibitem{Blumensath-Davies:2008}
{\sc Thomas Blumensath and Mike~E. Davies}, {\em Iterative thresholding for
  sparse approximations}, Journal of Fourier Analysis and Applications, 14
  (2008), pp.~629--654.

\bibitem{Bonettini2015New}
{\sc Silvia Bonettini and Marco Prato}, {\em New convergence results for the
  scaled gradient projection method}, Inverse Problems, 31 (2015), p.~095008.

\bibitem{Bruckstein-Donoho-Elad:2009}
{\sc Alfred~M. Bruckstein, David~L. Donoho, and Michael Elad}, {\em From sparse
  solutions of systems of equations to sparse modeling of signals and images},
  SIAM Review, 51 (2009), pp.~34--81.

\bibitem{Cand2006Sparsity}
{\sc Emmanuel Candes and Justin Romberg}, {\em Sparsity and incoherence in
  compressive sampling}, Inverse Problems, 23 (2006), pp.~969--985.

\bibitem{Candes-Wakin-Boyd:JFAA:08}
{\sc Emmanuel~J. Candes, Michael~B. Wakin, and Stephen~P. Boyd}, {\em Enhancing
  sparsity by reweighted $\ell_1$ minimization}, Journal of Fourier Analysis
  and Applications, 14 (2008), pp.~877--905.

\bibitem{Chartrand:IEEE-Letter:07}
{\sc Rick Chartrand}, {\em Exact reconstruction of sparse signals via nonconvex
  minimization}, IEEE Signal Processing Letters, 14 (2007), pp.~707--710.

\bibitem{Chen-Lu:penalty2014}
{\sc Xiaojun Chen, Zhaosong Lu, and Ting~Kei Pong}, {\em Penalty methods for
  non-lipschitz optimization}, Eprint Arxiv, arXive:1409.2558v2,  (2015).

\bibitem{Davis1997Adaptive}
{\sc Geoffrey Davis, Stephane Mallat, and Marco Avellaneda}, {\em Adaptive
  greedy approximations}, Constructive Approximation, 13 (1997), pp.~57--98.

\bibitem{Donoho:2006}
{\sc David~L. Donoho}, {\em Compressed sensing}, IEEE Transactions on
  Information Theory, 52 (2006), pp.~1289--1306.

\bibitem{Fan2002Variable}
{\sc Jianqing Fan and Runze Li}, {\em Variable selection via nonconcave
  penalized likelihood and its oracle properties}, Journal of the American
  Statistical Association, 96 (2002), pp.~1348--1360.

\bibitem{Figueiredo2008Gradient}
{\sc Mario~A.T. Figueiredo, Robert~D. Nowak, and Stephen~J. Wright}, {\em
  Gradient projection for sparse reconstruction: Application to compressed
  sensing and other inverse problems}, IEEE Journal of Selected Topics in
  Signal Processing, 1 (2008), pp.~586--597.

\bibitem{Ito-Kunisch:2014}
{\sc Kazufumi Ito and Karl Kunisch}, {\em A variational approach to sparsity
  optimization based on lagrange multiplier theory}, Inverse Problems, 30
  (2014), p.~015001.

\bibitem{jiao-jin-lvv:acha2016}
{\sc Yuling Jiao, Bangti Jin, and Lu~Xiliang}, {\em A primal dual active set
  with continuation algorithm for the $\ell_0$-regularized optimization
  problem}, Applied and Computational Harmonic Analysis, 39 (2015).

\bibitem{Lu:2014}
{\sc Zhaosong Lu}, {\em Iterative reweighted minimization methods for
  regularized unconstrained nonlinear programming}, Mathematical Programming,
  147 (2014), pp.~277--307.

\bibitem{Natarajan1995Sparse}
{\sc Balas~Kausik Natarajan}, {\em Sparse approximate solutions to linear
  systems}, SIAM Journal on Computing, 24 (1995), pp.~227--234.

\bibitem{Needell-Tropp:2009}
{\sc Deanna Needell and Joel~A. Tropp}, {\em Cosamp: Iterative signal recovery
  from incomplete and inaccurate samples}, Applied and Computational Harmonic
  Analysis, 26 (2009), pp.~301--321.

\bibitem{Nikolova:2013}
{\sc Mila Nikolova}, {\em Description of the minimizers of least squares
  regularized with $\ell_0$-norm. uniqueness of the global minimizer}, SIAM
  Journal on Imaging Sciences, 6 (2013), pp.~904--937.

\bibitem{Nikolova:ACHA2016}
\leavevmode\vrule height 2pt depth -1.6pt width 23pt, {\em Relationship between
  the optimal solutions of least squares regularized with $\ell_0$-norm and
  constrained by k-sparsity}, Applied and Computational Harmonic Analysis,
  (2016).

\bibitem{Nikolova-Ng-Zhang-Ching:2008}
{\sc Mila Nikolova, Michael~K. Ng, Shuqin Zhang, and Wai-Ki Ching}, {\em
  Efficient reconstruction of piecewise constant images using nonsmooth
  nonconvex minimization}, SIAM Journal on Imaging Sciences, 1 (2008),
  pp.~2--25.

\bibitem{Nocedal-Wright:book}
{\sc Jorge Nocedal and Stephen~J. Wright}, {\em Numerical Optimization},
  Springer, 2006.

\bibitem{Rockafellar2004Variational}
{\sc R.~Tyrrell Rockafellar and Roger J.~B. Wets}, {\em Variational analysis},
  Springer, 2004.

\bibitem{Tropp:2006}
{\sc Joel~A. Tropp}, {\em Just relax: Convex programming methods for
  identifying sparse signals in noise}, IEEE Transactions on Information
  Theory, 52 (2006), pp.~1030--1051.

\bibitem{Tropp-Gilbert:2007}
{\sc Joel~A. Tropp and Anna~C. Gilbert}, {\em Signal recovery from random
  measurements via orthogonal matching pursuit}, IEEE Transactions on
  Information Theory, 53 (2007), pp.~4655--4666.

\bibitem{Wright-Nowak-Figueiredo:IEEESP-09}
{\sc Stephen~J. Wright, Robert~D. Nowak, and Mario~A.T. Figueiredo}, {\em
  Sparse reconstruction by separable approximation}, IEEE Transactions on
  Signal Processing, 57 (2009), pp.~2479 --2493.

\bibitem{Zhang:ANNALS:2010}
{\sc Cun-Hui Zhang}, {\em Nearly unbiased variable selection under minimax
  concave penalty}, Annals of Statistics, 38 (2010), pp.~894--942.


\end{thebibliography}

\end{document}